\documentclass[11pt,reqno]{amsart}
\usepackage{amsthm,amssymb,amsmath,enumerate}
\usepackage{graphicx}

\newtheorem{theorem}{Theorem}
\newtheorem*{theorem*}{Theorem}
\newtheorem{lemma}{Lemma}
\newtheorem{corollary}{Corollary}
\newtheorem*{corollary*}{Corollary}
\newtheorem{proposition}{Proposition}
\newtheorem{claim}{Claim}
\newtheorem*{claim*}{Claim}

\newtheorem*{MoserLemma}{Moser's Lemma}
\newtheorem*{theorem2}{Theorem 2A}
\newtheorem*{theorem3}{Theorem 3A}
\newtheorem*{theorem2prime}{Theorem 2B}
\newtheorem*{theorem3prime}{Theorem 3B}

\theoremstyle{definition}

\newtheorem{definition}{\bf Definition}
\newtheorem*{definition*}{\bf Definition}

\newtheorem{remark}{\bf Remark}
\newtheorem*{remark*}{\bf Remark}

\newtheorem*{example*}{\bf Example}

\newcommand{\loc}{{\rm loc}}

\newcommand{\mydiv}{{\rm div\,}}

\begin{document}

\title[Parabolic equations with singular (form-bounded) vector fields]{Heat kernel bounds for parabolic equations with singular (form-bounded) vector fields}

\author{D.\,Kinzebulatov}

\address{Universit\'{e} Laval, D\'{e}partement de math\'{e}matiques et de statistique, 1045 av.\,de la M\'{e}decine, Qu\'{e}bec, QC, G1V 0A6, Canada}

\email{damir.kinzebulatov@mat.ulaval.ca}

\thanks{
The research of D.K. was supported by grants from NSERC and FRQNT}

\author{Yu.\,A.\,Sem\"{e}nov}

\address{University of Toronto, Department of Mathematics, 40 St.\,George Str, Toronto, ON, M5S 2E4, Canada}

\email{semenov.yu.a@gmail.com}



\begin{abstract}
We consider Kolmogorov operator $-\nabla \cdot a \cdot \nabla + b \cdot \nabla$ with measurable uniformly elliptic matrix $a$ and prove Gaussian lower and upper bounds on its heat kernel under minimal assumptions on the vector field $b$ and its divergence ${\rm div\,}b$. More precisely, we prove:

(1) Gaussian lower bound, provided that ${\rm div\,}b \geq 0$, and $b$ is in the class of form-bounded vector fields (containing e.g.\,the class $L^d$, the weak $L^d$ class, as well as some vector fields that are not even in $L_{\rm loc}^{2+\varepsilon}$, $\varepsilon>0$); in these assumptions, the Gaussian upper bound is in general invalid;

(2) Gaussian upper bound, provided that $b$ is form-bounded, and the positive part of ${\rm div\,}b$ is in the Kato class; in these assumptions, the Gaussian lower bound is in general invalid;

(3) Gaussian upper and lower bounds, provided that $b$ is form-bounded, ${\rm div\,}b$ is in the Kato class;

(4) A priori Gaussian upper and lower bounds, provided that $b$ is in a large class containing the class of form-bounded vector fields,  ${\rm div\,}b$ is in the Kato class.
\end{abstract}

\maketitle

\section{Introduction and main results}

\label{intro_sect}

\noindent\textbf{1.~}The subject of this paper is Gaussian lower and upper bounds on the heat kernel $u(t,x;s,y)$, $t>s$, of the parabolic equation
\begin{equation}
\label{eq}
(\partial_t - \nabla_x \cdot a \cdot \nabla_x + b \cdot \nabla_x)u(t,x)=0 \quad \text{ on } \mathbb [0,\infty[\times \mathbb R^d, \; d \geq 3,
\end{equation}
\begin{equation}
\tag{$H_{\sigma,\xi}$}
\begin{array}{c}
a=a^*:\mathbb R^d \rightarrow \mathbb R^d \otimes \mathbb R^d, \\
 \sigma I \leq a(x) \leq \xi I \quad \text{ for a.e. } x \in \mathbb R^d \quad \text{ for constants $0<\sigma<\xi<\infty$},
\end{array}
\end{equation}
under general assumptions on $$b=(b_i)_{i=1}^d:\mathbb R^d \rightarrow \mathbb R^d \quad \text{ and } \quad {\rm div\,}b=\sum_{i=1}^d\nabla_{x_i}b_i$$ that admit critical-order singularities.

The problem of existence of sharp elementary bounds on the heat kernel of the parabolic equation \eqref{eq}, and the ensuing  regularity properties of the heat kernel,
have been studied for several decades, with the principal breakthrough due to E.\,De Giorgi  \cite{DG} and J.\,Nash \cite{N} who treated the case $b=0$. D.\,G.\,Aronson \cite{A} established a two-sided Gaussian bound on $u$ of \eqref{eq} in the case $b=b_1+b_2$ with $|b_1| \in L^p$, $p>d$ and $b_2 \in L^\infty$. It was demonstrated in \cite{S,KiS}, that the Gaussian bounds on $u$ depend, in fact, on a much finer integral characteristics of $b$ than $\|b_1\|_p$, $p>d$ and $\|b_2\|_\infty$, that is, on the Nash norm of $b$, which allows to treat vector fields that may not even be in $L^{2+\varepsilon}_{\loc}$ for a given $\varepsilon>0$. This line of research is motivated, in particular, by the desire to find the quantitative relationship between the integral characteristics of $a$, $b$ and the regularity properties of $u$.
Another motivation for studying discontinuous $a$ and singular (i.e.\,locally unbounded) $b$ comes from physical applications. These applications make relevant assumptions on the integral properties of ${\rm div\,}b$. In presence of such assumptions, as is well known, one should be able to treat considerably more singular $b$. This is the subject of this paper.
More precisely, below we show that the heat kernel of \eqref{eq} satisfies:

1) a Gaussian lower bound, provided that
$$
b \in \mathbf{F}_\delta\;\;  (\text{the class of form-bounded vector fields, see below) \quad for some } \delta \in]0,4\sigma^2[
$$
and
$$
{\rm div\,}b \geq 0
$$
(let us note that in general a Gaussian upper bound is not valid under these assumptions);

\smallskip

2) a Gaussian upper bound, provided that
$
b \in \mathbf{F}_\delta
$
for some $\delta<\infty$,
and
$$
({\rm div\,}b)_+ \in \mathbf{K}^d_{\nu}\;\;  (\text{the Kato class, see below}) \quad \text{ for } \nu \text{ sufficiently small}
$$
(in general, there is no Gaussian lower bound under these assumptions);

\smallskip

3) two-sided Gaussian bound, provided that
$$
b \in \mathbf{F}_\delta \quad \text{ for some } \delta<\infty,
$$
and
$$
{\rm div\,}b \in \mathbf{K}^d_{\nu}\;\;  \text{ for } \nu \text{ sufficiently small}.
$$

\smallskip

4) a priori two-sided Gaussian bound, provided that 
$$
b \in \mathbf{MF}_\delta \;\; (\text{the multiplicative class of form-bounded vector fields, see below)}
$$
for some $\delta<\infty$, and
$
{\rm div\,}b \in \mathbf{K}^d_{\nu}$ with $\nu$ sufficiently small.

\medskip

The closest to ours results were obtained in the case $a=I$ in \cite{LZ}, see detailed comparison below.
It should be added that in the case ${\rm div\,}b=0$ one can  relax the assumptions on $b$ even further (although then the corresponding bounds become, in general, non-Gaussian), see \cite{Z2}. See also \cite{Z2_,QX,QX2}.

\medskip

In what follows, $L^p \equiv L^p(\mathbb R^d,dx)$, $L^p_{\loc} \equiv L^p_{\loc}(\mathbb R^d,dx)$.
 
\begin{definition*}
A vector field $b:\mathbb R^d \rightarrow \mathbb R^d$ is said to be form-bounded if $|b| \in L^2_\loc$ and there exist constants $\delta > 0$ and $c(\delta) \geq 0$ such that 
\begin{equation*}
\|b f\|_2^2 \leq \delta \|\nabla f\|_2^2 + c(\delta) \|f\|_2^2, \quad f \in W^{1,2} \equiv W^{1,2}(\mathbb R^d)
\end{equation*}
or, shortly,
$$
|b|^2 \leq \delta (-\Delta) + c(\delta) \quad (\text{in the sense of quadratic forms})
$$
(written as $b \in \mathbf{F}_{\delta}=\mathbf{F}_\delta(-\Delta)$).
\end{definition*}

\begin{definition*}
A vector field $b:\mathbb R^d \rightarrow \mathbb R^d$ is in the multiplicative class of form-bounded vector fields if $|b| \in L^1_\loc$ and there exist constants $\delta>0$ and $c(\delta) \geq 0$ such that 
$$
|\langle bf,f\rangle| \leq \delta \sqrt{\|\nabla f\|^2_2+c(\delta)\|f\|^2_2}\|f\|_2, \quad f \in W^{1,2}
$$
(written as $b \in \mathbf{MF}_\delta$).
\end{definition*}

Here and below,
$$
\langle f\rangle:=\int_{\mathbb R^d}fdx, \quad \langle f,g\rangle:=\langle f\bar{g}\rangle.
$$

\begin{definition*}
A potential $V\in L^1_\loc$ is said to be in the Kato class if there exist constants $\nu>0$ and $\lambda \equiv \lambda(\nu) \geq 0$ such that 
$$
\|(\lambda-\Delta)^{-1}|V|\|_\infty \leq \nu$$
(written as $V \in \mathbf{K}^d_\nu$).
\end{definition*}

We first comment on the classes $\mathbf{K}^d_\nu$ and $\mathbf{F}_\delta$. Both classes have been studied in the literature. 
The class $\mathbf{K}^d_\nu$ was introduced in 1961 by M.S.\,Birman \cite[Sect.\,2]{B}  as an elementary sufficient condition for the form-boundedness:
$$\text{if } V \equiv |b|^2 \in \mathbf{K}^d_\delta, \text{ then } b \in \mathbf{F}_\delta.
$$
Indeed, let $f\in L^2$, then
\begin{align*}
\|(\lambda-\Delta)^{-\frac{1}{2}}V^\frac{1}{2}f\|_2^2 & =\langle V^\frac{1}{2}f,(\lambda-\Delta)^{-1}V^\frac{1}{2}f\rangle\\
&=\langle\langle V^\frac{1}{2}(x)f(x)[(\lambda-\Delta)^{-1}(y,x)]^\frac{1}{2}[(\lambda-\Delta)^{-1}(x,y)]^\frac{1}{2}V^\frac{1}{2}(y)\overline{f(y)}\rangle_x\rangle_y\\
&\leq \big[\langle\langle(\lambda-\Delta)^{-1}(y,x)V(x)|f(y)|^2\rangle_x\rangle_y\big]^\frac{1}{2}\big[\langle\langle(\lambda-\Delta)^{-1}(x,y)V(y)|f(x)|^2\rangle_y\rangle_x\big]^\frac{1}{2}\\
&=\langle |f|^2(\lambda-\Delta)^{-1} V\rangle\leq \|(\lambda-\Delta)^{-1}V\|_\infty\|f\|_2^2.
\end{align*}

Let us mention some examples.

If $V  \in L^p + L^\infty$, $p>\frac{d}{2},$ then $V \in \mathbf{K}^d_\nu$ with arbitrarily small $\nu$. For every $\varepsilon>0$ there exist $V \in \mathbf{K}^d_\nu$ such that $V \not\in L^{1+\varepsilon}_{\loc}$.

If $$|b| \in L^d+L^\infty,$$ then $b \in \mathbf{F}_\delta$ with $\delta$ that can be chosen arbitrarily small (via the Sobolev Embedding Theorem). The class $\mathbf{F}_\delta$ also contains vector fields having critical-order singularities, 
e.g.\,
$$
b(x)=\pm \sqrt{\delta}\frac{d-2}{2}|x|^{-2}x \in \mathbf{F}_\delta \quad \text{with $c(\delta)=0$}
$$
(by the Hardy inequality $\frac{(d-2)^2}{4}\||x|^{-1}f\|_2^2 \leq \|\nabla f\|_2^2$, $f \in W^{1,2}$). More generally, $\mathbf{F}_\delta$ contains vector fields $b$ with $|b|$ in the weak $L^d$ space, the Campanato-Morrey class, the Chang-Wilson-Wolff class.  We refer to \cite[Sect.\,4]{KiS2} for more examples of form-bounded vector fields and a detailed discussion of class $\mathbf{F}_\delta$.

\medskip

The class $\mathbf{MF}_\delta$ is the largest. It contains  the class of weakly form-bounded fields $\mathbf{F}^{\scriptscriptstyle \frac{1}{2}}_\delta$, which consists of the vector fields $b$ such that $|b| \in L^1_{\loc}$ and, for some $\lambda=\lambda(\delta) \geq 0$,
$$
\||b|^{\frac{1}{2}}f\|^2_2 \leq \delta \|(\lambda-\Delta)^{\frac{1}{4}}f\|^2_2, \quad f \in \mathcal W^{\frac{1}{2},2} \;\;(\text{the Bessel potential space}).
$$
(The class $\mathbf{F}^{\scriptscriptstyle \frac{1}{2}}_\delta$ provides $\mathcal W^{1+\frac{1}{q},p}$-regularity theory for the operator $-\Delta + b \cdot \nabla$ for $p$ large and $q>p$, see \cite{KiS2} for details.) Indeed, for $f\in W^{1,2}$,
\begin{align*}
|\langle bf,f\rangle| &\leq \langle|b|f,f\rangle \leq \delta \langle (\lambda-\Delta)^{\frac{1}{2}}f,f \rangle \leq \delta \|(\lambda-\Delta)^{\frac{1}{2}}f\|_2\|f\|_2 \\
& = \delta\sqrt{\|\nabla f\|_2^2 + \lambda\|f\|_2^2}\|f\|_2.
\end{align*}
Note that, by interpolation, $\mathbf{F}_\delta \subsetneq \mathbf{F}^{\scriptscriptstyle \frac{1}{2}}_\delta$, furthermore, there exist $b \in \mathbf{F}^{\scriptscriptstyle \frac{1}{2}}_\delta$ such that $|b| \not \in L^{1+\varepsilon}_{\loc}$, $\varepsilon>0$ (and so these vector fields are in $\mathbf{MF}_\delta$). 
Another example is: if $$ b=\nabla_{x_i}\mathsf{f} \;\;(\text{in the sense of distributions}) \quad \text{ where $\mathsf{f}:\mathbb R^d \rightarrow \mathbb R^d$, $|\mathsf{f}| \in L^\infty$},$$ then, using the integration by parts, one has $b \in \mathbf{MF}_\delta$ with $\delta = 2\|\mathsf{f}\|_\infty$, $c(\delta)=0$.

The class $\mathbf{MF}_\delta$ (with $|b|$ in the LHS) was introduced in \cite{S2} as a class providing two-sided Gaussian bound on the heat kernel of $-\nabla \cdot a \cdot \nabla + b \cdot \nabla$ in the case ${\rm div\,}b=0$.

\medskip

\noindent\textbf{2.~}First, we will establish \textit{a priori} bounds on the heat kernel of $-\nabla \cdot a \cdot \nabla + b \cdot \nabla$, i.e.\,assuming additionally that $a$, $b$ are $C^\infty$ smooth, $|b|$ and ${\rm div\,}b$ are bounded. These bounds depend only on the dimension $d$, the ellipticity constants $\sigma$, $\xi$, the form-bound $\delta$ of $b$ (or the multiplicative form-bound $\delta$ of $b$) and the Kato relative bound $\nu$ of ${\rm div\,}b$, but not on the smoothness of $a$, $b$. 

To treat general measurable $a \in (H_{\sigma,\xi})$ and $b \in \mathbf{F}_\delta$ 
with ${\rm div\,}b \in \mathbf{K}^d_\nu$, 
we fix the following smooth approximations of $a$, $b$.
Set  $$a_{\varepsilon_1}:=E_{\varepsilon_1}a, \quad \varepsilon_1>0,$$
where $E_\varepsilon f:= e^{\varepsilon\Delta}f$ ($\varepsilon>0$), the De Giorgi mollifier of $f$.
It is easily seen that $a_{\varepsilon_1}$ are $C^\infty$ smooth and belong to $(H_{\sigma,\xi})$ for all $\varepsilon_1>0$.
We define
$$
b_\varepsilon:=E_\varepsilon b.
$$
In Section \ref{apost_sect0} we show that the vector fields $b_\varepsilon \in [L^\infty]^d$, $C^\infty$ smooth and are in $\mathbf{F}_\delta$ with the same  $c(\delta)$. (The proof of an analogous result for $b \in  \mathbf{MF}_\delta$ is given in Section \ref{proof_prime}.) Moreover, if ${\rm div\,}b \in \mathbf{K}^d_\nu$, then, for all $\varepsilon>0$, ${\rm div\,}b_\varepsilon \in L^\infty$, $C^\infty$ smooth and
$$
{\rm div\,}b_\varepsilon = ({\rm div\,}b)_\varepsilon \in \mathbf{K}^d_\nu 
$$
 with the same $\lambda=\lambda(\nu)$, see Section \ref{kato_rem_sect}. This choice of a regular approximation of $b$ is dictated by the need to control both $b_\varepsilon$ and ${\rm div\,}b_\varepsilon$ at the same time. (A straightforward approach of using cut-off functions to construct $b_\varepsilon$ leads to, generally speaking, loss of control over the Kato relative bound of ${\rm div\,}b_\varepsilon$. On the other hand, since $b \in \mathbf{F}_\delta$ does not entail $|b| \in L^2 + L^\infty$, the fact that $b_\varepsilon$ defined as above are  bounded requires justification.
A careful choice of appropriate smooth approximation of $b$ is needed even if $a=I$.

In what follows, we put
$$
\Lambda_{\varepsilon_1,\varepsilon}:=-\nabla \cdot a_{\varepsilon_1} \cdot \nabla + b_\varepsilon \cdot \nabla$$ 
with domain $D(\Lambda_{\varepsilon_1,\varepsilon})=W^{2,p}$ for $p$ that will be clear from the context.

\medskip

\noindent\textbf{3.~}We now state the main results of this paper in detail.
Put
 $$k_\mu(t,x,y)\equiv k(\mu t,x,y):=(4\pi\mu t)^{-\frac{d}{2}}e^{ - \frac{|x-y|^2}{4\mu t}}, \quad \mu>0.$$

\begin{theorem}[Lower bound]
\label{second_thm}
Let $d \geq 3$. Assume that $b \in \mathbf{F}_\delta=\mathbf{F}_\delta(-\Delta)$ for some $0<\delta<4\sigma^2$, and
$$
{\rm div\,}b \geq 0 \quad \text{(in the sense of tempered distributions).}
$$ 
Then, for each $p\in ]\frac{2}{2-\sqrt{\sigma^{-2}\delta}},\infty[$, the limit 
$$
s{\mbox-}L^p\mbox{-}\lim_{\varepsilon \downarrow 0} \lim_{\varepsilon_1 \downarrow 0}e^{-t\Lambda_{\varepsilon_1,\varepsilon}} \quad (\text{locally uniformly in $t \geq 0$})
$$
exists and determines in $L^p$ a positivity preserving $L^\infty$-contraction, quasi contraction $C_0$ semigroup, say, $e^{-t\Lambda}$. The operator $\Lambda$ is an appropriate operator realization of the formal operator $-\nabla \cdot a \cdot \nabla + b \cdot \nabla$ in $L^p$.

The semigroup $e^{-t\Lambda}$ is a semigroup of integral operators. Its integral kernel $u(t,x;s,y)=e^{-(t-s)\Lambda}(x,y)$  ($\equiv$ the heat kernel of $\Lambda$) satisfies the Gaussian lower bound
\begin{equation*}
\tag{{\rm LGB}}
\label{i_LGB}
  c_1 k_{c_2}(t-s;x-y) e^{-c_0 (t-s)} \leq u(t,x;s,y)
\end{equation*}
for a.e.\,$x, y \in \mathbb{R}^d$ and $0 \leq s < t < \infty,$ with constants $c_0 \geq 0$ and $c_i>0$ ($i=1,2$) that depend only of $d,\sigma,\xi,\delta,c(\delta)$. If $c(\delta)=0$ then $c_0=0$.
\end{theorem}

\setcounter{theorem}{3}

\begin{remark*} 1. In Theorem \ref{second_thm}, the Gaussian lower bound holds without any integrability assumptions on ${\rm div\,}b$, while the Gaussian upper bound is invalid. To the best of our knowledge, this is the first result of this type. 

2. Let us illustrate the fact that in the assumptions of Theorem \ref{second_thm} the heat kernel in general  does not satisfy a Gaussian upper bound.
Let $u(t,x;s,y)$ be the heat kernel of the operator $-\Delta + b \cdot \nabla$ with $b(x)=\sqrt{\delta}\frac{d-2}{2}|x|^{-2}x \in \mathbf{F}_\delta$, so ${\rm div}\,b=\sqrt{\delta}\frac{(d-2)^2}{2}|x|^{-2}$ is positive. If $\delta<4$, then $u(t,x;s,y)$ satisfies the two-sided bound
$$
c_1k_{c_2}(t-s;x-y)\varphi_{t-s}(y) \leq u(t,x;s,y) \leq c_3k_{c_4}(t-s;x-y)\varphi_{t-s}(y),
$$
where a positive singular weight $\varphi_{t} \in C^2(\mathbb R^d - \{0\})$  is uniformly bounded away from zero on $\mathbb R^d$ and satisfies $$\varphi_{t}(y) = |t^{-\frac{1}{2}}y|^{-\sqrt{\delta}\frac{d-2}{2}} \quad \text{ for }|y| \leq t^{\frac{1}{2}},\;t>0,$$
see \cite{MSS, MSS2}, see also \cite[Sect.\,4]{MS2}.
\end{remark*}

 Set $${\rm div\,}b_+:=0\vee{\rm div\,}b, \quad {\rm div\,}b_-:={\rm div\,}b_+-{\rm div\,}b.$$
In the next theorem we relax the assumptions ``${\rm div}b_-=0$'' and ``$\delta<4\sigma^2$'' of Theorem \ref{second_thm}, but impose a condition on ${\rm div\,}b_+$. 

\begin{theorem2}[Upper bound]
Let $d \geq 3$. Assume that

{\rm (\textit{1})}
$b \in \mathbf{F}_\delta$ for some $\delta<\infty$.

{\rm (\textit{2})}  ${\rm div\,}b_- \in L^1_\loc$ and $e^{\varepsilon\Delta}{\rm div\,}b_- \in L^\infty$ for each $\varepsilon>0$. 

{\rm (\textit{3})} ${\rm div\,}b_+ \in \mathbf{K}^d_\nu$ for some small $\nu$ dependent on $d,\sigma,\xi,\delta$.

Then the limit 
$$
s{\mbox-}L^2\mbox{-}\lim_{\varepsilon \downarrow 0} \lim_{\varepsilon_1 \downarrow 0}e^{-t\Lambda_{\varepsilon_1,\varepsilon}} \quad (\text{locally uniformly in $t \geq 0$}),
$$
exists and determines a positivity preserving $L^\infty$-contraction, quasi bounded $C_0$ semigroup of integral operators, say, $e^{-t\Lambda}$. Its integral kernel $u(t,x;s,y)$ satisfies the Gaussian upper bound
\begin{equation}
\tag{{\rm UGB}}
\label{i_UGB}
u(t,x;s,y) \leq c_3 k_{c_4}(t-s;x-y) e^{c_5 (t-s)}
\end{equation}
for a.e.\,$x, y \in \mathbb{R}^d$ and $0 \leq s < t < \infty$,
with constants  $c_3, c_4$ dependent on $d,\sigma,\xi,\delta,\nu$ and $c_5$ on $c(\delta),\lambda(\nu)$. If $c(\delta)=\lambda(\nu)=0$, then $c_5=0$.
\end{theorem2}

\begin{remark*} 1. In the assumptions of Theorem 2A, the heat kernel in general does not satisfy a Gaussian lower bound. For instance, the heat kernel $ u(t,x;s,y)$ of $-\Delta - b \cdot \nabla$ with $b(x)=\sqrt{\delta}\frac{d-2}{2}|x|^{-2}x \in \mathbf{F}_\delta$ satisfies
$$
c_1k_{c_2}(t-s;x-y)\psi_{t-s}(y) \leq u(t,x;s,y) \leq c_3k_{c_4}(t-s;x-y)\psi_{t-s}(y)
$$
with positive bounded weight $\psi_t(y)$ that vanishes at $y=0$, see \cite{MSS, MSS2}.

2. One can provide a number of sufficient conditions for the assumption ``$e^{\varepsilon\Delta}{\rm div\,}b_- \in L^\infty$ for each $\varepsilon>0$" to hold. For example, this assumption is satisfied if ${\rm div\,}b_- \in L^1 + L^\infty$ or if $\mathbf{1}_{B^c(0,R)}{\rm div\,}b_-$ is form-bounded
$$
\langle \mathbf{1}_{B^c(0,R)}{\rm div\,}b_-,|f|^2\rangle \leq \kappa \langle |\nabla f|^2 \rangle + c(\kappa,R) \langle |f|^2\rangle, \quad f \in W^{1,2} \quad \text{for some  $\kappa, R<\infty$},
$$
where $B^c(0,R):=\mathbb R^d - B(0,R)$.
(Indeed, we represent ${\rm div\,}b_-=\mathbf{1}_{B(0,R)}{\rm div\,}b_- + \mathbf{1}_{B^c(0,R)}{\rm div\,}b_-$, where the first term is in $L^1$ and so clearly $e^{\varepsilon\Delta}\mathbf{1}_{B(0,R)}{\rm div\,}b_- \in L^\infty \cap C^\infty$, while
$e^{\varepsilon\Delta} \mathbf{1}_{B^c(0,R)}{\rm div\,}b_- \in L^\infty \cap C^\infty$ by
repeating the proof of Claim \ref{claim1} below.) 
\end{remark*}

\begin{theorem2prime}[Upper bound]
Let $d \geq 3$. Assume that

{\rm (\textit{1})}
$b \in \mathbf{MF}_\delta$ for some $\delta<\infty$.

{\rm (\textit{2})}  ${\rm div\,}b_- \in L^1_\loc$ and $e^{\varepsilon\Delta}{\rm div\,}b_- \in L^\infty$ for each $\varepsilon>0$. 

{\rm (\textit{3})} ${\rm div\,}b_+ \in \mathbf{K}^d_\nu$ for some small $\nu$ dependent on $d,\sigma,\xi,\delta$.

Then $u_{\varepsilon_1,\varepsilon}(t,x;s,y) \equiv e^{-(t-s)\Lambda_{\varepsilon_1,\varepsilon}}(x,y)$ satisfy, for all $\varepsilon_1$, $\varepsilon>0$, the Gaussian upper bound
\begin{equation}
\tag{{\rm UGB}'}
\label{i_UGBprime}
u_{\varepsilon_1,\varepsilon}(t,x;s,y) \leq c_3 k_{c_4}(t-s;x-y) e^{c_5 (t-s)}
\end{equation}
for a.e.\,$x, y \in \mathbb{R}^d$ and $0 \leq s < t < \infty$,
with constants  $c_3, c_4$ dependent on $d,\sigma,\xi,\delta,\nu$ and $c_5$ on $c(\delta),\lambda(\nu)$ (but not on $\varepsilon_1$, $\varepsilon$). If $c(\delta)=\lambda(\nu)=0$, then $c_5=0$.
\end{theorem2prime}

Armed with the upper bound \eqref{i_UGBprime}, one can construct a limiting heat kernel using a standard argument appealing to the weak compactness in the space of measures and the Radon-Nikodym Theorem.
If $a=I$, the semigroups converge strongly in $L^2$ as in Theorem 2A (following closely the corresponding part of the proof of Theorem 2A). 

\medskip

In the next theorem we impose a more restrictive condition on ${\rm div\,}b_-$ than in Theorem 2A.

\begin{theorem3}[Two-sided bound]
Let $d \geq 3$. Assume that

{\rm (\textit{1})}
$b \in \mathbf{F}_\delta$ for some $\delta<\infty$.

{\rm (\textit{2})} $|{\rm div\,}b| \in \mathbf{K}^d_\nu$ for some small $\nu$ dependent on $d,\sigma,\xi,\delta$. 

Then the heat kernel $u(t,x;s,y)$ satisfies the two-sided Gaussian bound
\[
c_1 k_{c_2}(t-s;x-y)e^{-c_0(t-s)}\leq u(t,x;s,y)\leq c_3 k_{c_4}(t-s;x-y)e^{c_5(t-s)}.
\]
for a.e.\,$x, y \in \mathbb{R}^d$ and $0 \leq s < t < \infty$,
with constants  $c_i>0$ $(i=1,2,3,4)$ dependent on $d,\sigma,\xi,\delta,\nu$ and $c_0,c_5$ on $c(\delta),\lambda(\nu)$. If $c(\delta)=\lambda(\nu)= 0$, then $c_0=c_5=0$. 
\end{theorem3}

\begin{corollary} In the assumptions of Theorem 3A the following is true. 
 \label{cor1}

{\rm(\textit{i})} For every $f \in L^2$, $v(t,\cdot):=e^{-t\Lambda}f(\cdot)$ is H\"{o}lder continuous  (possibly after redefinition on a measure zero set in $\mathbb R^d \times \mathbb R^d$), i.e.\,for every $0<\alpha<1$ there exist constants $C<\infty$ and $\beta \in ]0,1[$ such that for all $z \in \mathbb R^d$, $s>R^2$, $0<R \leq 1$
$$
|v(t,x)-v(t',x')| \leq C
\|v\|_{L^\infty([s-R^2,s] \times \bar{B}(z,R))}
\biggl(\frac{|t-t'|^{\frac{1}{2}} + |x-x'|}{R} \biggr)^\beta
$$
for all $(t,x)$, $(t',x') \in [s-(1-\alpha^2)R^2,s] \times \bar{B}(z,(1-\alpha) R)$.

Furthermore, if $v \geq 0$, then it satisfies the Harnack inequality: Let $0<\alpha<\beta<1$, then there exists a constant $K=K(d,\sigma,\xi,\delta,\nu,\alpha,\beta)<\infty$ such that for all $(s,x) \in ]R^2,\infty[ \times \mathbb R^d$, $0<R \leq 1$ one has
$$
v(t,y) \leq K v(s,x)
$$
for all $(t,y) \in [s-\beta R^2,s-\alpha^2 R^2] \times \bar{B}(x,\delta R)$. 

\medskip

{\rm (\textit{ii})} The conservation of probability property:
$$
\langle u(t,x;s,\cdot)\rangle=1 \quad \text{ for all $x \in \mathbb R^d$, $t>s$.}
$$

\medskip

{\rm (\textit{iii})} 
$$e^{-(t-s)\Lambda_{C_u}}f(x):=\langle u(t,x;s,\cdot)f(\cdot)\rangle, \quad t>s, \quad f \in C_u$$
is a Feller semigroup on $C_u$, the space of bounded uniformly continuous functions on $\mathbb R^d$.
\end{corollary}

We first establish (\textit{i}) for $v_{\varepsilon_1,\varepsilon}(t,x)=e^{-t\Lambda_{\varepsilon_1,\varepsilon}}f(x)$, then  apply the Arzel\`{a}-Ascoli Theorem and use the convergence $e^{-t\Lambda}=s{\mbox-}L^2\mbox{-}\lim_{\varepsilon \downarrow 0} \lim_{\varepsilon_1 \downarrow 0}e^{-t\Lambda_{\varepsilon_1,\varepsilon}}$.
In turn, the proof of (\textit{i}) for $v_{\varepsilon_1,\varepsilon}$ repeats the argument in \cite[Sect.\,3]{FS}, which appeals to the ideas of E.\,De Giorgi \cite{DG} and uses \eqref{i_LGB}, \eqref{i_UGB}. 
The proof of (\textit{ii}) and (\textit{iii}) is a standard consequence of  \eqref{i_UGB},  the approximation result and the H\"{o}lder continuity of bounded solutions in (\textit{i}).

\begin{theorem3prime}[Two-sided bound]
Let $d \geq 3$. Assume that

{\rm (\textit{1})}
$b \in \mathbf{MF}_\delta$ for some $\delta<\infty$.

{\rm (\textit{2})} $|{\rm div\,}b| \in \mathbf{K}^d_\nu$ for some small $\nu$ dependent on $d,\sigma,\xi,\delta$. 

Then the heat kernel $u_{\varepsilon_1,\varepsilon}(t,x;s,y) \equiv e^{-(t-s)\Lambda_{\varepsilon_1,\varepsilon}}(x,y)$ satisfies, for all $\varepsilon_1$, $\varepsilon>0$, the two-sided Gaussian bound
\[
c_1 k_{c_2}(t-s;x-y)e^{-c_0(t-s)}\leq u_{\varepsilon_1,\varepsilon}(t,x;s,y)\leq c_3 k_{c_4}(t-s;x-y)e^{c_5(t-s)}.
\]
for a.e.\,$x, y \in \mathbb{R}^d$ and $0 \leq s < t < \infty$,
with constants  $c_i>0$ $(i=1,2,3,4)$ dependent on $d,\sigma,\xi,\delta,\nu$ and $c_0,c_5$ on $c(\delta),\lambda(\nu)$ (but not on $\varepsilon_1$, $\varepsilon$). If $c(\delta)=\lambda(\nu)= 0$, then $c_0=c_5=0$. 
\end{theorem3prime}

In the assumptions of Theorem 3B an a priori analogue of Corollary \ref{cor1} holds (i.e.\,with constants independent of $\varepsilon_1$, $\varepsilon$).

\begin{remark}
The proofs of heat kernel bounds in Theorems \ref{second_thm}, 2A (at least at the a priori level) can be extended, with minimal changes, to time-dependent coefficients. That is, let
$$
a=a^*:[0,\infty[ \times \mathbb R^d \rightarrow \mathbb R^d \otimes \mathbb R^d, \quad  \sigma I \leq a(t,x) \leq \xi I \quad \text{ for a.e. } (t,x) \in [0,\infty[ \times \mathbb R^d;
$$
we replace $\mathbf{F}_\delta$ by the class of time-dependent form-bounded vector fields
 $b:[0,\infty[ \times \mathbb R^d \rightarrow \mathbb R^d$, i.e.\,$|b| \in L^2_{\loc}([0,\infty[ \times \mathbb R^d)$ and there exists a constant $\delta>0$ such that
\begin{equation*}
\int_0^\infty\|b(t)f(t)\|_2^2 dt \leq \delta \int_0^\infty\|\nabla f(t)\|_2^2dt + \int_0^\infty g(t)\|f(t)\|_2^2dt
\end{equation*}
for some $g=g_\delta$ satisfying $\int_s^t g(\tau)d\tau \leq c_\delta\sqrt{t-s}$ (to obtain global in time bounds), for all $f \in L^1_{\loc}([0,\infty[,  W^{1,2})$;
the Kato class condition  in Theorem 2A is replaced with its time-dependent counterpart, see \cite{Z3}. Moreover, Theorem 2B also admits extension to time-dependent coefficients:
\begin{align*}
\int_0^\infty|\langle b(t)f(t),f(t)\rangle|dt  \leq \delta \int_0^\infty\|\nabla f(t)\|_2 \|f(t)\|_2 dt + \int_0^\infty g(t)\|f(t)\|_2^2 dt
\end{align*}
for all $f \in L^1_{\loc}([0,\infty[, W^{1,2})$ 
for some constant $\delta$ and a function $g$ satisfying the same assumptions as above.
\end{remark}

\bigskip

\noindent\textbf{4.~Existing results. }
In \cite{LZ} the authors constructed a weak heat kernel for $-\Delta + b \cdot \nabla$ satisfying Gaussian upper or lower bound under the following assumptions:

a) For the Gaussian upper bound: $b \in \mathbf{F}_\delta$ for some $0<\delta<\infty$, ${\rm div\,}b_+ \in \mathbf{K}^d_\nu$ for $\nu$ sufficiently small, and $|{\rm div\,}b|$ is form-bounded,
\begin{equation}
\label{div_LZ}
\tag{$\ast$}
\langle |{\rm div\,}b|,|f|^2\rangle \leq \kappa \langle |\nabla f|^2 \rangle + c(\kappa) \langle |f|^2\rangle, \quad f \in W^{1,2},
\end{equation}
with form-bound $\kappa<2$. 

We emphasize that according to our Theorem 2A, even in the special case $a=I$, the Gaussian upper bound on the heat kernel $e^{-t\Lambda}(x,y)$  is valid without the extra condition \eqref{div_LZ}.

\smallskip

b) For the Gaussian lower bound:  $b \in \mathbf{F}_\delta$ for some $0<\delta<\infty$ and ${\rm div\,}b=0$.

By Theorem \ref{second_thm}, even in the case $a=I$, their condition ``${\rm div\,}b=0$'' is relaxed to ``${\rm div\,}b \geq 0$'' albeit at expense of requiring $\delta<4\sigma^2$. 

c) They also proved the Gaussian lower bound on the heat kernel of $-\Delta +b \cdot \nabla$  defined via the Cameron-Martin-Girsanov formula, assuming that
$$|b|^2 \in \mathbf{K}^d_\delta \text{ and } b \in \mathbf{K}^{d+1}_{\nu} \quad (\equiv \|(\lambda-\Delta)^{-\frac{1}{2}}|b|\|_\infty \leq \nu)$$ with some   $\delta<\infty$, $c(\delta)\geq 0$, and $\nu<\infty$, $\lambda=\lambda(\nu) \geq 0$, thus refining the result in \cite{Z1} where the two-sided Gaussian bound on the heat kernel of $-\Delta +b \cdot \nabla$ is proved assuming only $b \in \mathbf{K}^{d+1}_\nu$ but with sufficiently small $\nu$ (in this regard, see also \cite{KiS}). Concerning semigroups defined via Cameron-Martin-Girsanov formula, see \cite{FK}.

\bigskip

\noindent\textbf{5.~On the proof of Theorem \ref{second_thm}.} We first establish a priori Gaussian lower bound, i.e.\,for smooth $a$, $b$. The proof is based on the method of J.\,Nash \cite{N} and its development in \cite{S}. The required a posteriori Gaussian lower bound then follows using approximation results in \cite{KiS2}.

\medskip

\noindent\textbf{6.~On the proof of Theorem 2A.} 
First, we establish Gaussian upper bound on the heat kernel of the auxiliary operator
$$
H^+=-\nabla \cdot a_{\varepsilon_1} \cdot \nabla + b_\varepsilon \cdot \nabla + E_\varepsilon {\rm div\,}b_+
$$
using only $E_\varepsilon{\rm div\,}b_+ \in L^\infty$ for every $\varepsilon>0$ (rather than stronger condition ${\rm div\,}b_+ \in \mathbf{K}^d_\nu$),
see Theorem \ref{aux_ub_thm}. The proof  uses J.\,Moser's iterations. Then the Gaussian upper bound on the heat kernel of $\Lambda_{\varepsilon_1,\varepsilon}=-\nabla \cdot a_{\varepsilon_1} \cdot \nabla + b_\varepsilon \cdot \nabla$ follows 
using the Duhamel formula, by considering $\Lambda_{\varepsilon_1,\varepsilon}$ as $H^+$ perturbed by the potential $-E_\varepsilon{\rm div\,}b_+ \in \mathbf{K}^d_\nu$, see Theorem \ref{apr_thm2}.  
Finally, we obtain the required (a posteriori) upper bound on the heat kernel of $\Lambda=-\nabla \cdot a \cdot \nabla + b \cdot \nabla$ by passing to the limit in $\varepsilon_1 \downarrow 0$ and then in $\varepsilon \downarrow 0$ (Proposition \ref{prop_approx}).

\medskip

\noindent\textbf{7.~On the proof of Theorem 2B.} The proof is obtained by modifying the proof of Theorem 2A, which amounts to estimating differently one term in the proof of the a priori upper bound of Theorem 2A.

\medskip

\noindent\textbf{8.~On the proof of Theorem 3A.} 
The Gaussian upper bound follows from Theorem 2A, so we only need to prove the Gaussian lower bound. First, we establish the lower bound on the heat kernel of the auxiliary operator
$$
H^-=-\nabla \cdot a_{\varepsilon_1} \cdot \nabla + b_\varepsilon \cdot \nabla - E_\varepsilon{\rm div\,}b_-
$$
using only $E_\varepsilon{\rm div\,}b_- \in L^\infty$ for every $\varepsilon>0$ (rather than ${\rm div\,}b_- \in \mathbf{K}^d_\nu$), see Theorem \ref{lb_div}.
The proof of the auxiliary lower bound of Theorem \ref{lb_div} is obtained by modifying the proof of Theorem \ref{second_thm} (Nash's method) to take advantage of the (a priori) Gaussian upper bound established in Theorem \ref{apr_thm2}.
Now, the lower bound on the heat kernel of $\Lambda_{\varepsilon_1,\varepsilon}=-\nabla \cdot a_{\varepsilon_1} \cdot \nabla + b_\varepsilon \cdot \nabla$ follows by considering $\Lambda_{\varepsilon_1,\varepsilon}$ as $H^-$ perturbed by $E_\varepsilon{\rm div\,}b_- \in \mathbf{K}^d_\nu$, and appealing to a pointwise inequality between the heat kernels of $\Lambda_{\varepsilon_1,\varepsilon}$, $H^-$ and $H^-_{p'}=-\nabla \cdot a_{\varepsilon_1} \cdot \nabla + b_\varepsilon \cdot \nabla  - p'E_\varepsilon{\rm div\,}b_-$, $p \geq 2$. 
The required (a posteriori) lower bound on the heat kernel of $\Lambda=-\nabla \cdot a \cdot \nabla + b \cdot \nabla$ follows using Proposition \ref{prop_approx}.

\medskip

 \noindent\textbf{9.~On the proof of Theorem 3B.} The Gaussian upper bound follows from Theorem 2B. To prove the Gaussian lower bound, we work with rather sophisticated regularization of Nash's $G$-functions, as in \cite{S2}. Once the bounds on the $G$-functions are established, we argue as in the proof of Theorem 3A.

\medskip

\tableofcontents

\bigskip

\section{Preliminaries}

\label{prelim_sect}

The following class of vector fields $b$ arises naturally in the study of operator $-\nabla \cdot a \cdot \nabla + b \cdot \nabla$.

Let $a \in (H_{\sigma,\xi})$. Let $A$ be the self-adjoint operator associated with the Dirichlet form $t[u,v]=\langle a\cdot\nabla u, \nabla v\rangle$, $u$, $v\in W^{1,2}$.

Given a vector field $b:\mathbb R^d \rightarrow \mathbb R^d$, we set $b_a:=\frac{b}{|b|}\sqrt{b\cdot a^{-1}\cdot b}$.

\begin{definition*}
A vector field $b:\mathbb R^d \rightarrow \mathbb R^d$ is said to be $A$-form-bounded  if $|b| \in L^2_\loc$ and there exist constants $\delta_a > 0$ and $c(\delta_a) \geq 0$ such that 
\begin{equation*}
\|b_a f\|_2^2 \leq \delta_a \|A^\frac{1}{2}f\|_2^2 + c(\delta_a) \|f\|_2^2, \quad f \in D(A^\frac{1}{2})= W^{1,2},
\end{equation*}
or, shortly, 
$$
|b_a|^2 \leq \delta_a A + c(\delta_a) \quad (\text{in the sense of quadratic forms}). 
$$
(written as $b \in \mathbf{F}_{\delta_a}(A)$).
\end{definition*}

It is easily seen that
\begin{equation*}
b \in \mathbf{F}_{\delta} \equiv \mathbf{F}_\delta(-\Delta) \quad \Longrightarrow \quad b \in \mathbf{F}_{\delta_a}(A),\;\;\delta_a=\sigma^{-2}\delta.
\end{equation*}

\bigskip

\section{Proof of Theorem \ref{second_thm}}

Since $b \in \mathbf{F}_\delta \equiv \mathbf{F}_\delta(-\Delta)$, $\delta<4\sigma^2$, we have $$b \in \mathbf{F}_{\delta_a}(A), \quad \delta_a<4.$$

First, we assume that $a \in (H_{\sigma,\xi})$ and $b$ are smooth, $b$ is bounded. 

\begin{definition}
A constant is said to be generic if it only depends on the dimension $d$, the ellipticity constants $\sigma$, $\xi$, the form-bound $\delta_a$ and the constant $c(\delta_a)$.
\end{definition}

The integral bound, the bounds on Nash's moment, entropy and the first (i.e.\,$\hat{G}$-) function contained in Sections \ref{int_bd_sect}-\ref{fbd_3_subsect_33}, which we use to prove the lower bound, appeared in \cite{S2} although there they were used for different purposes. Since they also play a crucial role in what follows, we include their proofs.

\subsection{Integral bound on the heat kernel of $-\nabla \cdot a \cdot \nabla + b \cdot \nabla$ for $b\in \mathbf{F}_{\delta_a}(A)$, $\delta_a<4$.} 

\label{int_bd_sect}

Set
\[
\Lambda = A + b \cdot \nabla, \quad A=-\nabla \cdot a \cdot \nabla.
\]
Let $U^{t,s}$ denote the solution of
\[ 
	\left\{ \begin{array}{rcl}
		-\frac{d}{d t} U^{t,s} f = \Lambda U^{t,s} f & , & 0 \leq s < t <\infty \\
	0 \leq f \in L^1 \cap L^\infty & 
	\end{array} \right.
	\tag{$CP_\Lambda$}
\]
in $L^p =L^p(\mathbb{R}^d), \;p \in [1,\infty[.$ 

Set $u(t):= U^{t,s}f.$ 
We have, for $p \in [p_c, \infty[$, $p_c=\frac{2}{2-\sqrt{\delta_a}}$,
\[
\big \langle \big(\frac{d}{d t} + \Lambda \big)u(t),u(t)^{p-1} \big  \rangle = 0.
\]
Setting $v:=u^{p/2}$, $w:=\langle v^2\rangle \equiv \|u(t,\cdot) \|^p_p$, $J:= \|A^{1/2}v \|^2_2,$ we have by quadratic estimates,
\begin{align*}
-\frac{d}{d t}w & = 2\big(\frac{2}{p^\prime}\|A^{1/2}v\|_2^2 + \langle \nabla v, b v \rangle\big), \;\; p^\prime =p/(p-1),
\end{align*}
\begin{align*}
|\langle \nabla v, b v\rangle | & \leq \langle  b \cdot a^{-1}\cdot b v, v \rangle^{1/2} \|A^{1/2} v \|_2^{1/2} \\
& \leq (\delta_a J + c(\delta_a) w)^{1/2} J^{1/2}\\
& (\text{we are using $b \in \mathbf{F}_{\delta_a}(A)$}) \\
& \leq \sqrt{\delta_a} J + (2 \sqrt{\delta_a})^{-1} c(\delta_a) w.
\end{align*}
\[
-\frac{d}{d t} w \geq 4 c_p J - \frac{1}{\sqrt{\delta_a}} c(\delta_a) w. \\
\tag{$\star$}
\]
where $c_p := \frac{1}{p^\prime}-\sqrt{\frac{\delta_a}{4}} = \frac{1}{p_c} -\frac{1}{p} \geq 0.$

 From $(\star)$ we obtain $-\frac{d}{d t} w \geq - \frac{1}{\sqrt{\delta_a}} c(\delta_a) w$, $w(t) \leq w(s) e^{\frac{1}{\sqrt{\delta_a}} c(\delta_a)(t-s)},$ or
\[
\|u(t,\cdot) \|_p \leq \|u(s,\cdot) \|_p e^{\frac{1}{p} \frac{c(\delta_a)}{ \sqrt{\delta_a}} (t-s)}.\\
\tag{$\star_a$}
\]
 In particular,
\[
\|u(t,\cdot)\|_\infty \leq \|u(s,\cdot)\|_\infty.\\
\tag{$\star_b$}
\]
Using the Nash inequality
\[
\|\nabla \psi \|_2^2 \geq c_N \| \psi \|_2^{2 +4/d} \|\psi\|_1^{-4/d},
\]
we obtain from $(\star)$ with $p=2 p_c,$ and so $c_p=\frac{1}{2p_c}$,
\[
-\frac{d}{d t} w \geq c_g w^{1+2/d} \|v\|_1^{-4/d} - \delta_a^{-1/2} c(\delta_a) w, \qquad c_g = 2\sigma c_N p_c^{-1}.
\]
Therefore
\[
\frac{d}{2} \frac{d}{d t} \big( w^{-2/d} \big) \geq c_g \|u\|_{p_c}^{-4p_c/d} - \delta_a^{-1/2} c(\delta_a) w^{-2/d}.
\]
This inequality is linear with respect to $\phi =w^{-2/d}.$ Thus setting $\mu(t)=\frac{2c(\delta_a)}{d \sqrt{\delta_a}} (t-s),$ we have, using $(\star_a),$
\begin{align*}
\frac{d}{d r} \big( e^{\mu(r)} \phi(r) \big) & \geq \frac{2c_g}{d} e^{\mu(r)} \|u(r,\cdot)\|_{p_c}^{-4p_c/d} \\
& \geq \frac{2c_g}{d} e^{-\mu(r)} \|u(s,\cdot)\|_{p_c}^{-4p_c/d},\\
e^{\mu(t)} \phi(t) & \geq \frac{2c_g}{d} \|u(s,\cdot)\|_{p_c}^{-4p_c /d} \int_s^t e^{-\mu(r)} d r\\
& \geq \frac{2c_g}{d} \|u(s,\cdot)\|_{p_c}^{-4p_c /d} e^{-\mu(t)} (t-s), \text{ and so }
\end{align*}
\[
\|u(t,\cdot)\|_{2p_c} \leq (d/(2c_g))^{d/4 p_c} e^\frac{c(\delta_a) (t-s)}{p_c \sqrt{\delta_a}} (t-s)^{-\frac{d}{2}\big( \frac{1}{p_c}-\frac{1}{2p_c} \big)} \|u(s,\cdot)\|_{p_c}.
\tag{$\star_c$}
\]

Applying  the Coulhon-Raynaud Extrapolation Lemma (Appendix \ref{app_B}) to 
\[
\big\|e^\frac{-c(\delta_a)t}{p_c \sqrt{\delta_a}}u(t,\cdot)\big\|_{2p_c} \leq (d/(2c_g))^{d/4 p_c} (t-s)^{-\frac{d}{2}\big( \frac{1}{p_c}-\frac{1}{2p_c} \big)} \big\|e^\frac{-c(\delta_a)s}{p_c \sqrt{\delta_a}} u(s,\cdot)\big\|_{p_c}
\]
and
\[
\big\|e^\frac{-c(\delta_a)t}{p_c \sqrt{\delta_a}}u(t,\cdot)\big\|_\infty \leq \big\|e^\frac{-c(\delta_a)s}{p_c \sqrt{\delta_a}}u(s,\cdot)\big\|_\infty,
\]
which is an immediate consequence of the inequalities $(\star_c)$ and $(\star_b)$, we obtain
\[
\|u(t,\cdot)\|_\infty \leq c e^\frac{c(\delta_a)(t-s)}{p_c \sqrt{\delta_a}} (t-s)^{-\frac{d}{2 p}} \|u(s,\cdot)\|_p \quad \forall p \in [p_c, \infty[\\
\tag{$\star\star$}
\]
with a generic constant $c$ (although it does not depend on $\xi$).

From $(\star\star)$ we immediately obtain the following
integral bound on the heat kernel $u(t,x;s,y)$ of $\Lambda$ ($\equiv$ the integral kernel of $U^{t,s}$):
\begin{equation}
\sup_{x \in \mathbb{R}^d} \langle u^{p^\prime}(t,x;s,\cdot)\rangle \leq c^{p'} e^\frac{p' c(\delta_a)(t-s)}{p_c \sqrt{\delta_a}} (t-s)^{-\frac{d}{2 (p-1)}}\quad \forall p \in [p_c, \infty[, \quad 0 \leq s < t <\infty.
\tag{$\circ$}
\label{int_bd_one_star}
\end{equation}

\subsection{Bounds on Nash's moment and entropy} 
\label{fbd_2_subsect}

Our assumptions on $b$ are as in Section \ref{int_bd_sect}.

In this section we assume $0 < t-s \leq 1$. (Let us note that if $c(\delta_a)=0$, then we can work over $0 \leq s<t <\infty$.)

Following J.\,Nash, define the entropy
\[
Q(s)\equiv Q(s;t,x) := - \langle u(t,x;s,\cdot) \log u(t,x;s,\cdot) \rangle
\]
and the moment
\[
M(s)\equiv M(s;t,x) := \langle |x-\cdot| u(t,x;s,\cdot) \rangle.
\]
The dynamic equation $\frac{d}{d s} u(t,x;s,\cdot) = \Lambda^* u(t,x;s,\cdot)$ (where $\Lambda^*=A-\nabla \cdot b$) and the conservation law $\langle u(t,x;s,\cdot) \rangle = 1$ yield
\[
- \frac{d}{d s}Q(s) = \mathcal{N}(s) + \langle b(\cdot) \cdot \nabla_\cdot u(t,x;s,\cdot) \rangle \equiv \bigg \langle \nabla u \cdot \frac{a}{u} \cdot \nabla u \bigg \rangle + \langle b \cdot \nabla u \rangle.
\]

\begin{proposition} \label{mb_thm} 
There exist generic constants $\mathbb{C}_1, c_\pm >0$ such that, for all $x \in \mathbb{R}^d$ and $0 < t-s \leq 1,$
\[
|Q(s) - \tilde{Q}(t-s)| \leq \mathbb{C}_1,\\
\tag{$\mbox{NEE}$}
\]
\[
c_- \sqrt{t-s} \leq M(s) \leq c_+ \sqrt{t-s},\\
\tag{$\mbox{NMB}$}
\]
where $\tilde{Q}(\tau) := \frac{d}{2} \log \tau.$
\end{proposition}

\begin{proof}[Proof of Proposition \ref{mb_thm}] We will repeatedly use $\langle u(t,x;s,\cdot) \rangle = 1$ and \eqref{int_bd_one_star}.

\begin{claim}
\label{_cl_1}
~ $Q(s) \geq \tilde{Q}(t-s) -C_{p_c}$ where $C_{p_c} =(p_c-1)\log \hat{c}, \; \hat{c}$ the constant from \eqref{int_bd_one_star}.
\end{claim}

\medskip

\begin{proof}[Proof of Claim \ref{_cl_1}]~ By Jensen's inequality and \eqref{int_bd_one_star} for $r=p_c,$
\begin{align*}
Q(s) & = -(r-1)\big \langle u \log u^\frac{1}{r-1} \big \rangle \\
& \geq - (r-1)\log \big \langle u^{r'}(t,x;s,\cdot) \big \rangle \\
& \geq -(r-1) \log \big (\hat{c} (t-s)^{-\frac{d}{2(r-1)}} \big).
\end{align*}

\end{proof}

\begin{claim}
\label{_cl_2}
~ $\mathcal{N}(s) \leq p_c \big(- Q^\prime (s) + c(\delta_a) /\sqrt{\delta_a} \big),$ where $Q^\prime (s)= \frac{d}{d s} Q(s).$
\end{claim}
\begin{proof}[Proof of Claim \ref{_cl_2}] Clearly, $\mathcal{N}(s)= - Q^\prime (s) - \langle b \cdot \nabla u \rangle,$ and
\begin{align*}
|\langle b \cdot \nabla u \rangle| & \leq \langle b \cdot a^{-1} \cdot b u \rangle^{1/2} \mathcal{N}^{1/2} \\
& \leq \big( \frac{\delta_a}{4} \mathcal{N} + c(\delta_a) \big)^{1/2} \mathcal{N}^{1/2} \\
& \leq \sqrt{\delta_a /4} \mathcal{N} + c(\delta_a) /\sqrt{\delta_a}.
\end{align*}
\end{proof}

\begin{claim}
\label{_cl_3}
\begin{align*}
M(s) \leq \sqrt{\xi} \bigg \{ \sqrt{2 p_c}\big(1+\sqrt{\delta_a/4}\big) \sqrt[4]{t-s} \bigg[ \frac{c(\delta_a)}{\sqrt{\delta_a}} \sqrt{t-s} & - \int_s^t \sqrt{t-\tau} Q^\prime(\tau) d \tau   \bigg]^{1/2}\\
& + \sqrt{t-s}(c(\delta_a))^{1/2} \bigg \}.
\end{align*}
\end{claim}
\begin{proof}[Proof of Claim \ref{_cl_3}] Clearly,
\begin{align*}
- M^\prime(s) & = -\langle |x-\cdot| \Lambda^*(\cdot) u(t,x;s,\cdot) \rangle \\
& =- \langle \nabla |x-\cdot| \cdot a \cdot \nabla u \rangle - \langle \nabla |x-\cdot|, b u \rangle \\
& \leq \bigg \langle \frac{x-\cdot}{|x-\cdot|} \cdot a u \cdot \frac{x-\cdot}{|x-\cdot|} \bigg \rangle^{1/2} \big( \mathcal{N}^{1/2} + \langle b \cdot a^{-1} \cdot b u \rangle^{1/2} \big).
\end{align*}
By $a \leq \xi I$, $\langle u \rangle = 1$ and $\langle b \cdot a^{-1} \cdot b u \rangle \leq \frac{\delta_a}{4} \mathcal{N} + c(\delta_a),$
\begin{align*}
- M^\prime(s) & \leq \sqrt{\xi} \bigg[\mathcal{N}^{1/2} + \big(\frac{\delta_a}{4} \mathcal{N} + c(\delta_a) \big)^{1/2} \bigg]\\
& \leq \sqrt{\xi} \bigg[(1+\sqrt{\delta_a/4}) \mathcal{N}^{1/2}(s) + \sqrt{c(\delta_a)} \bigg].
\end{align*}
Since $M(t)=0$ and $0<t-s \leq 1$,
\begin{align*}
M(s) & \leq \sqrt{\xi} \bigg[(1+\sqrt{\delta_a/4}) \int_s^t (\sqrt{t-\tau}\mathcal{N}(\tau))^{1/2} \frac{d \tau}{\sqrt[4]{t-\tau}} + (t-s) \sqrt{c(\delta_a)}  \bigg]\\
& \leq \sqrt{\xi} \bigg[(1+\sqrt{\delta_a/4}) \bigg(\int_s^t \frac{d \tau}{\sqrt{t-\tau}}\bigg)^{1/2} \bigg(\int_s^t \sqrt{t-\tau}\mathcal{N}(\tau) d \tau \bigg)^{1/2} + \sqrt{t-s}(c(\delta_a))^{1/2} \bigg]\\
& = \sqrt{\xi} \bigg[(1+\sqrt{\delta_a/4}) \sqrt{2}\sqrt[4]{t-s} \bigg(\int_s^t \sqrt{t-\tau}\mathcal{N}(\tau) d \tau \bigg)^{1/2} + \sqrt{t-s}(c(\delta_a))^{1/2} \bigg]. 
\end{align*}
By Claim \ref{_cl_2},
\begin{align*}
\int_s^t \sqrt{t-\tau}\mathcal{N}(\tau) d \tau & \leq p_c \int_s^t\bigg( \frac{\sqrt{t-\tau}}{\sqrt{\delta_a}} c(\delta_a) - \sqrt{t-\tau} Q^\prime(\tau) \bigg) d \tau \\
& \leq p_c \bigg( \frac{\sqrt{t-s}}{\sqrt{\delta_a}}c(\delta_a) - \int_s^t \sqrt{t-\tau}Q^\prime(\tau) d \tau \bigg).
\end{align*}
\end{proof}

 \begin{claim}
\label{_cl_4}
 \[
 - \int_s^t \sqrt{t-\tau}Q^\prime(\tau) d \tau \leq \sqrt{t-s}\big( Q(s) - \tilde{Q}(t-s) + C_{p_c} + d \big).
 \]
 \end{claim}

 Claim \ref{_cl_4} follows easily from Claim \ref{_cl_1} using integration by parts. 

\begin{claim}
\label{_cl_5}
\[
M(s) \leq \sqrt{{\xi}} \bigg( K_1 \sqrt{Q(s)-\tilde{Q}(t-s)+C_{p_c}+d} + K_2 (c(\delta_a))^{1/2} \bigg)\sqrt{t-s},
\]
where $K_1= \sqrt{2 p_c}(1+\sqrt{\delta_a/4})$ and $K_2 =1+\frac{K_1}{\sqrt{\delta_a}}.$  
\end{claim}

Claim \ref{_cl_5} is a simple corollary of Claim \ref{_cl_3} and Claim \ref{_cl_4}. 

\begin{claim}
\label{_cl_6}
There is a constant $c(d)<\infty$ such that
\[
e^{Q(s)/d} \leq c(d) M(s).
\]
\end{claim}

Claim \ref{_cl_6} follows from $\langle u \rangle = 1$ via the inequality $u \log u \geq - \mu u -e^{-1-\mu}$ for all real $\mu.$

Claim \ref{_cl_5} and Claim \ref{_cl_6} combined yield
\begin{claim}
\label{_cl_7}
\[
e^{\frac{2}{d}[Q(s)-\tilde{Q}(t-s)]} \leq 2 c(d)^2 \xi \bigg( K^2_1 \big[ Q(s)-\tilde{Q}(t-s)+C_{p_c}+d \big] + K^2_2 c(\delta_a) \bigg). 
\]
\end{claim}
Claim \ref{_cl_7} implies that, for all $0 < t-s \leq 1$ and all $x \in \mathbb{R}^d$ there is a generic constant $\mathbb{C}$ such that $Q(s)-\tilde{Q}(t-s) \leq \mathbb{C}.$ Taking into account Claim \ref{_cl_1}, Claim \ref{_cl_5} and Claim \ref{_cl_6} we arrive at $\mbox{(NEE)}$ and $\mbox{(NMB)}.$  
\end{proof}

\subsection{$\hat{G}$-bound} 

\label{fbd_3_subsect_33}

In what follows, $0 <t-s \leq 1$.
Define Nash's $\hat{G}$-function
\[
\hat{G}(s) 
:= \langle k_\beta(t-s,o-\cdot) \log u(t,x;s,\cdot) \rangle, \quad o = \frac{x+y}{2}
\]
for all $x,y \in \mathbb{R}^d$ such that $2 |x-y| \leq \sqrt{\beta(t-s)}$, where $\beta > \xi$ is a constant whose value we will be specified below.

The proof of the next proposition works under more general assumptions than in Theorem \ref{second_thm}, i.e.\,we may assume that $b$ satisfies the assumptions of Section \ref{int_bd_sect}.

\begin{proposition}
\label{prop_G2}
There exist generic constants $\beta$ and $\mathbb{C}$ such that 
$$
\hat{G}(t_s) \geq - \tilde{Q}(t-t_s) - \mathbb{C}, \qquad t_s=\frac{t+s}{2}.
$$
\end{proposition}

\begin{proof}[Proof of Proposition \ref{prop_G2}] 
Our proof of the $\hat{G}$-bound follows in general Nash's original proof and relies on the conservation law, the $M$-bound proved in Proposition \ref{mb_thm}, the Spectral gap inequality, the geometry of the euclidean space (i.e.\,the rate of growth of the volume of euclidean ball) and the integral bound \eqref{int_bd_one_star}. 

Let $\varepsilon>0$. Set $U(s):=u(t,x;s,\cdot)+\varepsilon$, $\varepsilon>0$ and put
\[
G(\tau)\equiv G_\varepsilon(\tau):= \langle k_\beta(t-s,o-\cdot) \log U(\tau) \rangle, \quad \tau \in [s,t_s].
\]
It suffices to carry out the proof for $G_\varepsilon$
since $\hat{G}(s) = \inf_{\varepsilon > 0}G_\varepsilon(s),$
Set
$$
\beta^* =  \frac{3}{2}\left(1+\frac{\delta_a}{4}\right)\xi.
$$

\medskip

\begin{claim}
\label{_claim1_}
For all $\tau \in [s, t_s]$ and $\beta  \geq \beta^*$
\[
- \bigg(G(\tau) + \tilde{Q}(t-\tau)\bigg)^\prime + \frac{3}{2}c(\delta_a)  \geq \frac{\sigma}{4\beta(t-s)}\big \langle k_\beta(t-s,o-\cdot) |\log U(\tau)-G(\tau)|^2 \big \rangle.
\]
\end{claim}

\begin{proof}[Proof of Claim \ref{_claim1_}]
Let $\mathcal{N}:=\langle \nabla \log U \cdot a \Gamma \cdot \nabla \log U \rangle, \; \mathcal{N}_0:=\langle \nabla \log \Gamma \cdot a \Gamma \cdot \nabla \log \Gamma \rangle,$ where $$\Gamma:= k_\beta(t-s,o-\cdot),$$ the Gaussian density, and let $b_a^2:=b\cdot a^{-1} \cdot b.$

The dynamic equation yields
\begin{align*}
-G^\prime(\tau)=-\langle \Gamma/U, u^\prime \rangle &= -\langle \Gamma/U, A U \rangle+ \langle \Gamma/U, \nabla \cdot b u\rangle \\
& = \mathcal{N} - \langle \nabla \Gamma \cdot a \cdot \nabla \log U \rangle - \langle b \cdot \nabla \Gamma, u/U \rangle + \langle \Gamma b \cdot \nabla \log U, u/U \rangle.
\end{align*}
By quadratic inequalities,
\[
|\langle \nabla \Gamma \cdot a \cdot \nabla \log U \rangle| \leq \mathcal{N}^{1/2} \mathcal{N}_0^{1/2}, \]

\[|\langle b \cdot \nabla \Gamma,u/U \rangle| \leq \mathcal{N}_0^{1/2} \langle b_a^2 \Gamma \rangle^{1/2},\]

\[|\langle \Gamma b \cdot \nabla \log U,u/U \rangle| \leq \mathcal{N}^{1/2} \langle b_a^2 \Gamma \rangle^{1/2}.
\]
Therefore,
\[
-G^\prime(\tau) \geq \frac{1}{2} \mathcal{N} - \frac{3}{2} \mathcal{N}_0 - \frac{3}{2} \langle b_a^2 \Gamma \rangle.
\]
Note that $\mathcal{N}_0 \leq \xi \langle (\nabla \Gamma)^2 /\Gamma \rangle =\frac{\xi d}{2 \beta(t-s)},$ and, since $b \in\mathbf{F}_{\delta_a}(A)$, 
\[
\langle b_a^2 \Gamma \rangle \leq \delta_a \langle \nabla \sqrt{\Gamma} \cdot a \cdot \nabla \sqrt{\Gamma} \rangle + c(\delta_a)\leq \frac{ \delta_a }{4} \mathcal N_0 + c(\delta_a).
\]
Thus, 
\[
-G^\prime(\tau) \geq \frac{1}{2} \mathcal{N}(\tau) - \frac{3}{4}\biggl(1+\frac{\delta_a}{4}\biggr) \frac{\xi d}{\beta(t-s)} - \frac{3}{2} c(\delta_a).
\]
Noticing that $-\frac{1}{t-s} \geq - \frac{1}{t-\tau}$ we have (for $\beta \geq \beta^* \equiv \frac{3}{2}(1+\frac{\delta_a}{4})\xi$)
\[
- \bigg( G(\tau) +\tilde{Q}(t-\tau) \bigg)^\prime + \frac{3}{2} c(\delta_a)  \geq \frac{1}{2} \mathcal{N}.
\]
At this point we use the Spectral gap inequality
\[
\big\langle \Gamma |\nabla \psi|^2 \big\rangle \geq \frac{1}{2 \beta (t-s)} \big \langle \Gamma |\psi - \langle \Gamma \psi \rangle |^2 \big \rangle
\]
obtaining
\[
- \bigg( G(\tau) +\tilde{Q}(t-\tau)\bigg)^\prime - \frac{3}{2} c(\delta_a) \geq \frac{\sigma}{4 \beta (t-s)} \big \langle \Gamma |\log U - G(\tau) |^2 \big \rangle.
\]
\end{proof}

\begin{claim}
\label{_claim10_}
Set $\Phi:= |\log U(t,x;\tau, \cdot) -G(\tau)|,$ $\tau \in [s, t_s].$ Let $\chi$ denote the indicator of the ball $B(o,\sqrt{\beta (t-s)}).$ There is a generic constant $c(\beta) > 0$ such that, for any $r \geq p_c,$
\[
- \bigg( G(\tau) +\tilde{Q}(t-\tau)\bigg)^\prime + \frac{3}{2}c(\delta_a) \geq c(\beta) (t-s)^{-1 + \frac{d(2-r)}{2(r-1)}} \big\langle \chi u^{r^\prime/2} \Phi \big \rangle^2.
\]
\end{claim}

\begin{proof}[Proof of Claim \ref{_claim10_}] Clearly, $\chi \Gamma \geq c_\beta (t-s)^{-d/2}\chi,$  $c_\beta =(4\pi \beta)^{-d/2}e^{-\frac{1}{4}}.$ Thus
\[
\big \langle \Gamma \Phi^2 \big \rangle \geq c_\beta (t-s)^{-d/2} \big \langle \chi \Phi^2 \big \rangle. 
\]
By H\"older's inequality, $\big \langle \chi \Phi^2 \big \rangle \geq \big \langle \chi u^{r^\prime/2} \Phi \big \rangle^2 / \big \langle u^{r^\prime} \big \rangle.$ By the integral bound \eqref{int_bd_one_star} in Section \ref{int_bd_sect},
$$\big \langle u^{r^\prime}(t,x;\tau,\cdot) \big \rangle \leq \hat{c}(t-\tau)^{-\frac{d}{2(r-1)}},$$ and by the inequality $t-\tau \geq \frac{t-s}{2},$
 \[
\big \langle \Gamma \Phi^2 \big \rangle \geq 2^{-\frac{d}{2(r-1)}} \frac{c_\beta}{\hat{c}}(t-s)^{\frac{d}{2}\frac{2-r}{r-1}} \big \langle \chi u^{r^\prime/2} \Phi \big \rangle^2.
 \]
Now Claim \ref{_claim10_} with $c(\beta)=\frac{\sigma c_\beta}{4\beta \hat{c}} 2^{-\frac{d}{2(r-1)}}$ follows from Claim \ref{_claim1_}.
\end{proof}

\begin{claim}
\label{_claim11_}
Set $\theta =r^\prime /2.$ Then
\[
\big \langle \chi u^\theta \Phi \big \rangle \geq \big \langle \chi u^\theta \big \rangle \bigg[-G +\theta^{-1} \log \frac{\big \langle \chi u^\theta \big \rangle}{\langle \chi \rangle} \bigg].
\]
\end{claim}

\begin{proof}[Proof of Claim \ref{_claim11_}]
By the definition of $\Phi$,
\[
\big \langle \chi u^\theta \Phi \big \rangle \geq \big \langle \chi u^\theta \log U \big \rangle - \big \langle \chi u^\theta \big \rangle G.\\
\tag{$\star$}
\]
Using the inequality $v \log v \geq -m v -e^{-1-m}, \;v\geq 0, \;m$ real, we have
\[
\big \langle \chi u^\theta \log U \big \rangle \geq \big \langle \chi u^\theta \log u \big \rangle \geq -m\theta^{-1}\big \langle \chi u^\theta \big \rangle -\theta^{-1}e^{-1-m} \langle \chi \rangle.
\]
Putting here $-1-m=\log \frac{\big \langle \chi u^\theta \big \rangle}{\langle \chi \rangle},$ it is seen that
\[
\big \langle \chi u^\theta \log U \big \rangle \geq \theta^{-1} \langle \chi u^\theta \big \rangle \log \frac{\big \langle \chi u^\theta \big \rangle}{\langle \chi \rangle}.
\]
Substituting the latter in $(\star)$ ends the proof.
\end{proof}

Now, if $0<\delta_a \leq 1,$ then $p_c=(1-\sqrt{\delta_a/4})^{-1} \leq 2$ and we can take $r=2$, in which case we can proceed directly to Claim \ref{_claim13_}. 
In the more interesting case $1 < \delta_a < 4,$ however, $r \geq p_c > 2$, and the next claim plays a crucial role.

\begin{claim}
\label{_claim12_}
Let $\hat{c}$ be the constant from the integral bound \eqref{int_bd_one_star} in Section \ref{int_bd_sect}:
$$\big \langle u^{r^\prime}(t,x;\tau,\cdot) \big \rangle \leq \hat{c}(t-\tau)^{-\frac{d}{2(r-1)}}, \quad r > 2.$$ 
Then, for all $\tau \in [s, t_s],$ 
 \[
\big \langle \chi u^\theta \big \rangle \geq c^\prime (t-s)^\frac{d(r-2)}{4(r-1)} \langle \chi u \rangle^\frac{r}{2}
 \]
 and
 \[
\big \langle \chi u^\theta \Phi \big \rangle \geq \big \langle \chi u^\theta \big \rangle \big[- G(\tau) - \tilde{Q}(t-\tau) + (r-1)\log \langle \chi u \rangle -c^{\prime \prime} \big],
 \]
 where $c^\prime = \big(2^\frac{d}{2(r-1)} \hat{c} \big)^\frac{2-r}{2}$ and $c^{\prime \prime} = \frac{(r-2)(r-1)}{r} \log \hat{c} + \frac{d(r-1)}{r} \log \big(2 \beta \omega_d^{2/d}\big).$
\end{claim}
 
\begin{proof} [Proof of Claim \ref{_claim12_}]
The first inequality follows from H\"older's inequality $$\big \langle \chi u^\theta \big \rangle \geq \langle \chi u \rangle^{r/2} \big \langle u^{r^\prime} \big \rangle^{(2-r)/2}$$ because $r > 2$ and $$\big \langle u^{r^\prime} \big \rangle^{(2-r)/2} \geq \hat{c}^\frac{2-r}{2}(t-s)^\frac{d(r-2)}{4(r-1)} 2^\frac{d(2-r)}{4(r-1)}. $$ The second inequality follows from the first one, Claim \ref{_claim11_} and the equality $\langle \chi \rangle = \omega_d (\beta (t-s))^{d/2}.$
\end{proof}

\begin{claim}
\label{_claim13_}
Fix a $\beta \geq \max(\beta^*, (4 c_+)^2),$ where $\beta^*$ and $c_+$ are defined before Claim \ref{_claim1_} and in Proposition \ref{mb_thm}, respectively. Then $\langle \chi u(t,x;\tau,\cdot) \rangle \geq \frac{1}{2}$ for all $\tau \in [s, t].$
\end{claim}
\begin{proof}[Proof of Claim \ref{_claim13_}] Recalling that $2 |x-y| \leq \sqrt{\beta(t-s)},$ and so $|o-\cdot| \leq 4^{-1} \sqrt{\beta(t-s)} + |x-\cdot|,$ we have
\begin{align*}
\langle (1-\chi)u \rangle & = \int_{|o-z| \geq \sqrt{\beta(t-s)}} u(t,x;\tau,z) d z \\
& \leq \bigg \langle \frac{|o-\cdot|}{\sqrt{\beta(t-s)}} u(t,x;\tau,\cdot) \bigg \rangle \\
& \leq \bigg \langle \frac{|o-x|+|x-\cdot|}{\sqrt{\beta(t-s)}} u(t,x;\tau,\cdot) \bigg \rangle \\
& \leq \frac{1}{4} + \frac{1}{\sqrt{\beta(t-s)}}M(\tau) \\
& \leq \frac{1}{4} + \frac{c_+ \sqrt{t-s}}{\sqrt{\beta(t-s)}} \leq \frac{1}{2}
\end{align*}
and hence $\langle \chi u \rangle = 1 - \langle (1-\chi)u \rangle \geq \frac{1}{2}.$
\end{proof}

Claim \ref{_claim12_} and Claim \ref{_claim13_} combined yield
\[
\big \langle \chi u^\theta \Phi \big \rangle \geq \big \langle \chi u^\theta \big \rangle [-G(\tau) -\tilde{Q}(t-\tau)-c_1 ],\\
\tag{$\star \star$}
\]
where $c_1 = c^{\prime \prime} + (r-1) \log 2 > 0$ with $r= \max(2,p_c).$

Using last inequality and Claim \ref{_claim10_}, we end the proof of Proposition \ref{prop_G2} as follows.
Set
\[
I_0 := - G(\tau) - \tilde{Q}(t-\tau) -\frac{3}{2} c(\delta_a) (t-\tau).
\]
If $I_0 \geq 2 c_1$ for all  $\tau \in [s, t_s],$ then $I_0 - c_1 \geq \frac{1}{2}I_0 > c_1 > 0,$ and so by Claim \ref{_claim10_} and then by Claim \ref{_claim12_} and $(\star \star)$,
\begin{align*}
\frac{d}{d \tau} I_0 & \geq c(\beta) (t-s)^{-1+\frac{d}{2}\frac{2-r}{r-1}} \big  \langle \chi u^\theta \big\rangle^2 \bigg[ I_0 -c_1 + \frac{3}{2} c(\delta_a) (t-\tau) \bigg]^2\\
& \geq c (t-s)^{-1} I_0^2 \qquad \text{ with } c= c(\beta)(c^\prime)^2 2^{-r-2},
\end{align*}
or $-\frac{d}{d \tau} I_0^{-1} \geq c (t-s)^{-1}.$ Integrating the latter over $[s,t_s]$ yields $I_0^{-1}(s) \geq \frac{c}{2}$, or $I_0(s) \leq \frac{2}{c}$, or
\[
G(s) \geq - \tilde{Q}(t-s) - \frac{3}{2}c(\delta_a)(t-s) - \frac{2}{c}.
\]
If $I_0 < 2 c_1$ for some $\tau \in [s, t_s],$ then by Claim \ref{_claim1_}, $\frac{d}{d \tau} I_0 \geq 0,$ and hence $G(s)+\tilde{Q}(t-s) +\frac{3}{2}c(\delta_a)(t-s) \geq G(\tau) + \tilde{Q}(t-\tau) + \frac{3}{2} c(\delta_a) (t-\tau) \geq -2 c_1.$
\end{proof}

\medskip

\subsection{$G$-bound for $-\nabla \cdot a \cdot \nabla + \nabla \cdot b$} 


Set $\Lambda_*=A + \nabla \cdot b$.
Let $U_*^{t,s}$ denote the solution of
\[ 
	\left\{ \begin{array}{rcl}
		-\frac{d}{d t} U_*^{t,s} f = \Lambda_*(t) U_*^{t,s} f & , & 0 < t-s \leq 1, \\
	0 \leq f \in L^1 \cap L^\infty & 
	\end{array} \right.
\]
in $L^p =L^p(\mathbb{R}^d), \;p \in [1,\infty[.$ 

Set $u_*(t):= U_*^{t,s}f$ and let $u_*(t,x;s,y)$ denote the heat kernel of $\Lambda_*$. We introduce
\[
Q(t)\equiv Q(t;s,y) := - \langle u_*(t,\cdot; s, y)) \log u_*(t,\cdot; s, y) \rangle,
\]
\[
M(t)\equiv M(t;s,y) := \langle |y-\cdot| u_*(t,\cdot; s, y) \rangle, \text{ and }
\]
\[
G(t)
:= \langle k_\beta(t-s,o -\cdot)\log u_*(t, \cdot; s,y) \rangle.
\]
We will need the integral bound 
\[
\langle u^{p'}_*(t,\cdot;s,y)\rangle\leq c_g(t-s)^{-\frac{d}{2(p-1)}}\quad \forall p\in [p_c,\infty[, \; 0 <t-s \leq 1. \tag{$\bullet$}
\]

\medskip

\begin{proof}[Proof of {\rm($\bullet$)}]
Clearly, $-\langle\frac{d}{dt}u_*(t),u_*^{q-1}(t)\rangle=\langle (A+\nabla\cdot b)u_*(t),u_*^{q-1}(t)\rangle$, $1<q\leq p_c'=\frac{2}{\sqrt{\delta_a}}$. Thus, setting $w:=\langle u_*q\rangle$, $v:=u_*^\frac{q}{2}$, $J:=\|A^\frac{1}{2}v\|_2^2$, we have
\[
-\frac{1}{q}\frac{d}{dt}w=\frac{2}{q'}\big(\frac{2}{q}J-\langle bv,\nabla v\rangle\big).
\]
Using assumption $b\in\mathbf{F}_{\delta_a}(A)$, $\delta_a<4$ we have by quadratic estimates, $|\langle bv,\nabla v\rangle |\leq J^\frac{1}{2}(\delta_a J+gw)^\frac{1}{2}$, and so
\[
-\frac{d}{dt}w\geq 2(q-1)\big[\big(\frac{2}{q}-\sqrt{\delta_a}\big)J-\frac{1}{2\sqrt{\delta_a}}c(\delta_a) w\big] \tag{$\ast$}.
\]
From $(\ast)$ we have $\|u_*(t)\|_q\leq \|f\|_q e^\frac{\|g\|_1}{q'\sqrt{\delta_a}}$. In particular, $\|u_*(t)\|_1\leq \|f\|_1$. Also from $(\ast)$ we obtain $\|u_*(t)\|_q\leq c_g(t-s)^{-\frac{d}{2q'}}\|f\|_1$, and by duality, $\|(U^{t,s}_*)^*\|_{q'\to\infty}\leq c_g(t-s)^{-\frac{d}{2q'}}$. Now, $(\bullet)$ is evident.
\end{proof}

\medskip

Armed with ($\bullet$), we repeat word by word the arguments from the previous section, arriving at the following proposition. 

\begin{proposition}
\label{prop_G3}
Let $\beta$ and $\mathbb{C}$ be (generic) constants defined in Proposition \ref{prop_G2} and Proposition \ref{mb_thm}, respectively. 
$$
G(t_s) \geq - \tilde{Q}(t_s-s) - \mathbb{C}, \qquad t_s=\frac{t+s}{2}
$$
for all $0 \leq t-s \leq 1$ and $ x,y \in \mathbb{R}^d$ such that $2 |x-y| \leq \sqrt{\beta(t-s)}.$
\end{proposition}

Similarly to Proposition \ref{prop_G2}, the proof works under more general assumptions than in Theorem \ref{second_thm}, i.e.\,we may assume that $b$ is as in Section \ref{int_bd_sect}.

\subsection{A priori lower bound}

Recall that $a \in (H_{\sigma,\xi})$ and $b$ are smooth, $b$ is bounded,  ${\rm div\,}b \geq 0$.

It is seen from the Duhamel formula that, since ${\rm div\,}b \geq 0$, $$u_*(t,x;s,y) \leq u(t,x;s,y), \quad 0 <t-s \leq 1.$$
We have
\[
u(t,x;s,y) \geq (4 \pi \beta (t-t_s))^{d/2} \langle k_\beta (t-t_s,o -\cdot) u(t,x;t_s,\cdot) u(t_s,\cdot;s,y) \rangle,
\]
and, for all $2 |x-y| \leq \sqrt{\beta(t-t_s)}$, due to Proposition \ref{prop_G2} and Proposition \ref{prop_G3},
\begin{align*}
\log u(t,x;s,y) & \geq \log (4 \pi \beta)^{d/2} + \tilde{Q}(t-t_s)\\
& + \langle k_\beta (t-t_s,o-\cdot) \log u(t,x;t_s,\cdot) \rangle + \langle k_\beta (t-t_s,o-\cdot) \log u_*(t_s,\cdot; s,y) \rangle \\
& \geq \log (4 \pi \beta)^{d/2} - \tilde{Q}(t-t_s)  -2 \mathbb{C} \\
& = - \tilde{Q}(t-s)  -2 \mathbb{C} + \log (8 \pi \beta)^{d/2},
\end{align*}
so a Gaussian lower bound for $u(t,x;s,y)$ follows but only for  $2 |x-y| \leq \sqrt{\beta(t-t_s)}$.
Now, the standard argument (``'small gains yield large gain'), see e.g.\,\cite[Theorem 3.3.4]{Da}, and the reproduction property of $u(t,x;s,y)$ give

\begin{theorem}
\label{lb_div_}
There exist generic constants $c_0 \geq 0$ and $c_1$, $c_2 > 0$ such that, for all $x,y \in \mathbb{R}^d$
\begin{equation*}
c_1 k_{c_2}(t-s,x-y) e^{-c_0(t-s)} \leq u(t,x;s,y)
\end{equation*}
for all $0 \leq s < t <\infty$.
\end{theorem}

We emphasize that the constants $c_1$, $c_2$ are generic, and thus do not depend on the smoothness of $a$, $b$, and the boundedness of $b$.

\subsection{A posteriori lower bound} 

\label{apost_sect0}

We now exclude the assumption of the smoothness of $a$, $b$, and the boundedness of $b$ by constructing a smooth bounded approximation of $b \in \mathbf{F}_\delta \equiv \mathbf{F}_\delta(-\Delta)$ that preserves the relative bound $\delta$ and the constant $c(\delta)$.

 Define
$$
b_\varepsilon:=E_\varepsilon b,
$$
where, recall, $E_\varepsilon f:= e^{\varepsilon\Delta}f$ ($\varepsilon>0$) denotes the De Giorgi mollifier of $f$.

\begin{claim} \label{claim1}
The following is true:

\smallskip

1. $b_\varepsilon\in [L^\infty\cap C^\infty]^d$. 

\smallskip

2.~$b_\varepsilon \in \mathbf{F}_\delta$ with the same $c(\delta)$ (thus, independent of $\varepsilon$). 
\end{claim}

\begin{proof}[Proof of Claim \ref{claim1}]
1.~Since $b_\varepsilon=E_{\varepsilon/2}E_{\varepsilon/2}b$, it suffices to only prove that $|b_\varepsilon| \in L^\infty$. We have by Fatou's Lemma,
\begin{align*}
|b_\varepsilon(x)| & \leq \liminf_n\big\langle e^{\varepsilon\Delta}(x,\cdot)\mathbf{1}_{B(0,n)}(\cdot) |b(\cdot)|\big\rangle \\
& \leq \liminf_n\big\langle e^{\varepsilon\Delta}(x,\cdot)\mathbf{1}_{B(0,n)}(\cdot)|b(\cdot)|^2\big\rangle^{\frac{1}{2}} \leq \big(\delta \big\langle \big|\nabla \sqrt{e^{\varepsilon\Delta}(x,\cdot)}\big|^2\big\rangle + c(\delta)\big)^{\frac{1}{2}},
\end{align*}
where $\big|\nabla_y \sqrt{ e^{\varepsilon\Delta}(x,y)}\big|=(4\pi\varepsilon)^{-\frac{d}{4}} \frac{|x-y|}{4\varepsilon}e^{-\frac{|x-y|^2}{8\varepsilon}} \leq C\varepsilon^{-\frac{d}{4}-\frac{1}{2}} e^{-\frac{c|x-y|^2}{\varepsilon}}$, and so $|b_\varepsilon| \in L^\infty$ for each $\varepsilon>0$. 

\smallskip

2. Indeed, $|b_\varepsilon|\leq\sqrt{E_\varepsilon |b|^2}$, and so
\begin{align*}
\|b_\varepsilon f\|_2^2 &\leq \langle E_\varepsilon |b|^2,|f|^2\rangle=\|b\sqrt{E_\varepsilon|f|^2}\|^2_2 \\
&\leq \delta\|\nabla\sqrt{E_\varepsilon|f|^2}\|_2^2+c(\delta)\|f\|_2^2, \quad f \in W^{1,2},
\end{align*}
where
\begin{align*}
\|\nabla\sqrt{E_\varepsilon|f|^2}\|_2 & =\big\|\frac{E_\varepsilon(|f||\nabla|f|)}{\sqrt{E_\varepsilon|f|^2}}\big\|_2\\
&\leq \|\sqrt{E_\varepsilon|\nabla |f||^2}\|_2=\|E_\varepsilon|\nabla |f||^2\|_1^\frac{1}{2}\\
& \leq\|\nabla|f|\|_2\leq \|\nabla f\|_2,
\end{align*}
i.e. $b_\varepsilon\in\mathbf{F}_{\delta}$. [The fact that $\|b\sqrt{E_\varepsilon|f|^2}\|_2<\infty$ follows from $\mathbf{1}_{\{|b|\leq n\}}b\in \mathbf{F}_\delta$, the inequality $\|\mathbf{1}_{\{|b|\leq n\}}b \sqrt{E_\varepsilon|f|^2}\|_2^2\leq \delta\|\nabla f\|_2^2+c(\delta)\|f\|_2^2$ and Fatou's Lemma].
\end{proof}

\begin{claim}
\label{claim1__}
${\rm div\,}b_\varepsilon \geq 0$. 
\end{claim}

\begin{proof}
Indeed, since ${\rm div\,}b \geq 0$ in the sense of tempered distributions, i.e.\,$\langle b,\nabla \varphi \rangle \leq 0$ for every $0 \leq \varphi \in \mathcal S$, we have
$
\langle b_\varepsilon,\nabla \varphi \rangle = \langle b, \nabla E_\varepsilon \varphi \rangle \leq 0,
$
as needed. 
\end{proof}

We are in position to complete the proof of Theorem \ref{second_thm}. 

By Claim \ref{claim1}, $b_\varepsilon \in \mathbf{F}_\delta$, $\delta<4\sigma^2$, and so $b_\varepsilon \in \mathbf{F}_{\delta_a}(A)$, $\delta_a=\sigma^{-2}\delta<4$. Thus, \cite[Theorems 4.2, 4.3]{KiS2} apply and yield that the limit 
$$
s{\mbox-}L^p\mbox{-}\lim_{\varepsilon \downarrow 0} \lim_{\varepsilon_1 \downarrow 0}e^{-t\Lambda_{\varepsilon_1,\varepsilon}} \quad (\text{locally uniformly in $t \geq 0$}), \quad p> \frac{2}{2-\sqrt{\sigma^{-2}\delta}},
$$
where $$
\Lambda_{\varepsilon_1,\varepsilon}:=-\nabla \cdot a_{\varepsilon_1} \cdot \nabla + b_\varepsilon \cdot \nabla, \quad a_{\varepsilon_1}:=E_{\varepsilon_1}a \in (H_{\sigma,\xi}),\quad D(\Lambda_{\varepsilon_1,\varepsilon})=W^{2,p}
$$
exists and determines in $L^p$ a positivity preserving $L^\infty$-contraction, quasi contraction $C_0$ semigroup of integral operators, say, $e^{-t\Lambda}$.

Next, by Claim \ref{claim1} and Claim \ref{claim1__}, Theorem \ref{lb_div_} applies to the heat kernel $e^{-t\Lambda_{\varepsilon_1,\varepsilon}}$ with constants $c_0$-$c_2$ independent of $\varepsilon_1$, $\varepsilon$. Therefore, for every pair of balls $B_1$, $B_2 \subset \mathbb R^d$ we have
$$
c_1 e^{-c_0 t} \langle \mathbf{1}_{B_1},e^{tc_2 \Delta}\mathbf{1}_{B_2} \rangle \leq \langle \mathbf{1}_{B_1},e^{-t\Lambda_{\varepsilon,\varepsilon_1}}\mathbf{1}_{B_2} \rangle.
$$
Now, passing to the limit in $\varepsilon_1$ and then in $\varepsilon$, we obtain
$$
c_1 e^{-c_0 t} \langle \mathbf{1}_{B_1},e^{tc_2 \Delta}\mathbf{1}_{B_2} \rangle \leq \langle \mathbf{1}_{B_1},e^{-t\Lambda}\mathbf{1}_{B_2} \rangle.
$$
Applying the Lebesgue Differentiation Theorem, we complete the proof of Theorem \ref{second_thm}.

\bigskip

\section{Proof of Theorem 2A}

\label{ub_sect1}

Recall that, by the assumption of Theorem 2A, $b \in \mathbf{F}_\delta$, $0<\delta<\infty$, and so
\begin{equation*}
b \in \mathbf{F}_{\delta_a }(A),\;\;\delta_a=\sigma^{-2}\delta<\infty.
\end{equation*}

Recall that a constant is called \textit{generic} if it only depends on the dimension $d$, the ellipticity constants $\sigma$, $\xi$, the form-bound $\delta_a$ and the constants $c(\delta_a)$.

We will first prove Theorem 2A for the smoothed out coefficients $a_{\varepsilon_1}$, $b_\varepsilon$ (see Theorem \ref{apr_thm2} below). By Claim \ref{claim1}, $b_\varepsilon$ are bounded and are in $\mathbf{F}_\delta$ with the same $c(\delta)$ (thus, independent of $\varepsilon$). 

\subsection{A remark on the approximation of Kato class potentials}

\label{kato_rem_sect}

Let $V \in \mathbf{K}^d_\nu$. Define 
$$
V_\varepsilon = E_\varepsilon V, \quad \varepsilon>0,
$$
where, recall, $E_\varepsilon f:= e^{\varepsilon\Delta}f$ ($\varepsilon>0$) denotes the De Giorgi mollifier of $f$. 
Below we will be interested in the case $$V={\rm div\,}b_+, \text{ so } V_\varepsilon= E_\varepsilon{\rm div\,}b_+.$$

\begin{claim} \label{claim2} 
\smallskip
1.~$V_\varepsilon \in \mathbf{K}^d_\nu$ with the same $\lambda=\lambda(\nu)$ (independent of $\varepsilon$),

2.~$V_\varepsilon \in L^\infty \cap C^\infty$.

\end{claim}

\begin{proof}[Proof of Claim \ref{claim2}] 

1.~By duality, it suffices to prove that $\||V_\varepsilon|(\lambda-\Delta)^{-1}f\|_1 \leq \nu\|f\|_1$, $f \in L^1$. We have $|V_\varepsilon| \leq |V|_\varepsilon$, and
\begin{align*}
\||V|_\varepsilon(\lambda-\Delta)^{-1}f\|_1 &= \||V|(\lambda-\Delta)^{-1}E_\varepsilon f\|_1 \\
&\leq \nu\|E_\varepsilon f\|_1 \leq \nu\|f\|_1.
\end{align*}

2.~Since $V_\varepsilon$, $\varepsilon>0$ are form-bounded,
$$
\langle |V_\varepsilon|,|f|^2\rangle \leq \nu \langle |\nabla f|^2 \rangle + c_\nu \langle |f|^2\rangle, \quad f \in W^{1,2}, \quad c_\nu=\lambda \nu,
$$
see the introduction, we can argue as in the proof of assertion 1 of Claim \ref{claim1}.
\end{proof}

\subsection{Upper bound for the auxiliary operator $-\nabla \cdot a_{\varepsilon_1} \cdot \nabla + b_\varepsilon \cdot \nabla + E_\varepsilon  {\rm div\,}b_+$}

\label{fbd_6_subsect}

Set $$A_{\varepsilon_1}=-\nabla \cdot a_{\varepsilon_1} \cdot \nabla,$$
and
$$H^+ = A_{\varepsilon_1} + b_\varepsilon \cdot \nabla + E_\varepsilon\mydiv b_+.$$
Let $H^{t,s}f$ denote the solution of
\[ 
	\left\{ \begin{array}{rcl}
		-\frac{d}{d t} H^{t,s} f = H^+ H^{t,s} f & , & 0 \leq s < t <\infty \\
	0 \leq f \in L^1 \cap L^\infty & 
	\end{array} \right.
	\tag{$CP_{H^+}$}
\]
in $L^p =L^p(\mathbb{R}^d), \;p \in [1,\infty[.$ 

Let $h(t,x;s,y)$ denote the heat kernel of $H^+$, that is, $H^{t,s}f=\langle h(t,x;s,\cdot)f(\cdot)\rangle$.

\begin{theorem}
\label{aux_ub_thm}
There exist generic constants $c_3,c_4 > 0$, $\omega\geq 0$ such that
\[
h(t,x;s,y) \leq  c_3 k_{c_4}(t-s,x-y)e^{-(t-s)\omega}
\tag{$\mbox{UGB}^{h_+}$}
\]
for all $0 \leq s<t <\infty$.
\end{theorem}

\begin{proof}[Proof of Theorem] $\mathbf{1}.$ Since $A_{\varepsilon_1} + b_\varepsilon \cdot \nabla + E_\varepsilon\mydiv b_+ = A_{\varepsilon_1} + \nabla \cdot b_\varepsilon + E_\varepsilon\mydiv b_-$ (where $E_\varepsilon\mydiv b_-(x)$ and $E_\varepsilon\mydiv b_+(x)$ are uniformly bounded in $x\in\mathbb{R}^d$ and smooth by the assumptions of Theorem 2A and Claim \ref{claim2}, respectively), 
\[
 \langle \hat{h} \rangle := \langle h(t , x ; s , \cdot) \rangle \leq 1 \text{ and }\langle h \rangle := \langle h(t , \cdot ; s , y) \rangle \leq 1 .
\]
Also, since $\mydiv b_\varepsilon=E_\varepsilon\mydiv b=E_\varepsilon\mydiv b_+-E_\varepsilon\mydiv b_-$,
\begin{align*}
& \langle H^+ h , h \rangle = J + \frac{1}{2} \langle E_\varepsilon| \mydiv b | , h^2 \rangle \geq J , \;\;J := \langle \nabla h \cdot a_{\varepsilon_1} \cdot \nabla h \rangle , \\
& \langle (H^+)^*(s) \hat{h} , \hat{h} \rangle = \hat{J} + \frac{1}{2} \langle E_\varepsilon |\mydiv b | ,\hat{h}^2 \rangle \geq \hat{J} , \;\;\hat{J} := \langle \nabla \hat{h} \cdot a_{\varepsilon_1} \cdot \nabla \hat{h} \rangle ,
\end{align*}
and so estimating $\| H^{t,s} \|_{ L^1 \to L^2}$ and $\| \big( H^{t,s} \big)^* \|_{L^1 \to L^2}$ by means of the Nash inequality, we obtain 
\[
h(t , x ; s , y) \leq c (t-s)^{-d/2} , \;\; c = c (d , \sigma) .
\tag{$\mbox{NIE}^{h_+}$} 
\]
Here $x,y \in \mathbb{R}^d$ and $0 \leq s < t <\infty$.

\medskip
 
$\mathbf{2}.$ In order to prove $(\mbox{UGB}^{h_+})$ we consider
\[ 
	\left\{ \begin{array}{rcl}
		-\frac{d}{d t} H_\alpha^{t,s} f = H_\alpha^+ H_\alpha^{t,s} f & , & 0 \leq s < t <\infty, \\
	0 \leq f \in L^1 \cap L^\infty & 
	\end{array} \right.
	\tag{$CP_{H_\alpha^+}$}
\]
in $L^p =L^p(\mathbb{R}^d), \;p \in [1,\infty[,$ where $H^{t,s}_\alpha := e^{\alpha \cdot x} H^{t,s} e^{-\alpha \cdot x}$ and 
\[
 H_\alpha^+:= e^{\alpha \cdot x}(\omega+ H^+) e^{-\alpha \cdot x} =\omega+ H^+ - \alpha \cdot b_\varepsilon - \alpha \cdot a_{\varepsilon_1} \cdot \alpha + \alpha \cdot a_{\varepsilon_1} \cdot \nabla + \nabla \cdot a_{\varepsilon_1} \cdot \alpha,
\] 
\[
\omega=\frac{c(\delta_a)}{2\delta_a}.
\]

To shorten notation, in the rest of this section we write $A \equiv A_{\varepsilon_1}$.

\begin{MoserLemma} There are generic constants $c, c_4$ such that, for all $0 \leq s < t < \infty,$
\[
\|H^{t,s}_\alpha \|_{2 \to \infty}, \|H^{t,s}_\alpha \|_{1 \to 2} \leq c (t-s)^{-d/4} e^{c_4 \alpha^2(t-s)}.
\]
\end{MoserLemma}

\begin{proof}[Proof of Lemma] We follow \cite[Sect.\,1]{FS}.
Set $u_\alpha(t) := H_\alpha^{t,s} f$, $v(t):= u_\alpha^{p/2}(t)$, $p\geq 2.$ Noticing that $\langle b_\varepsilon \cdot \nabla u_\alpha, u_\alpha^{p-1} \rangle = \frac{2}{p} \langle \nabla v, b_\varepsilon v \rangle = -\frac{1}{p} \langle v^2, E_\varepsilon \mydiv b \rangle$, we have by the dynamic equation
\begin{align*}
-\frac{1}{p} \frac{d}{d t} \langle v^2(t) \rangle & = \omega\|v(t)\|_2^2 + \frac{4}{p p^\prime} \|A^{1/2} v(t) \|_2^2 +\frac{1}{p^\prime} \langle v^2(t),E_\varepsilon\mydiv b_+ \rangle +\frac{1}{p} \langle v^2(t),E_\varepsilon\mydiv b_- \rangle \\
& - \frac{2(p-2)}{p} \langle \alpha \cdot a_{\varepsilon_1} \cdot \nabla v(t), v(t) \rangle -\langle \alpha \cdot b_\varepsilon, v^2(t) \rangle -\langle \alpha \cdot a_{\varepsilon_1} \cdot \alpha, v^2(t) \rangle.
\end{align*}
By quadratic estimates and by ($b_\varepsilon \in \mathbf{F}_\delta$, see Claim \ref{claim1} $\Rightarrow$ $b_\varepsilon\in\mathbf{F}_{\delta_a}(A)$, $\delta_a=\sigma^{-2}\delta$),
\begin{align*}
-\frac{1}{p} \frac{d}{d t}\|v\|_2^2 & \geq \frac{4}{p p^\prime}(1-\kappa -\gamma \delta_a) \|A^{1/2} v \|_2^2 + \bigg[\omega- \frac{4}{p p^\prime} \gamma c(\delta_a)\bigg] \|v\|_2^2 \\
& - \bigg[1 +\frac{1}{4 \gamma}\frac{p p^\prime}{4} + \frac{1}{4 \kappa}\frac{p p^\prime}{4} 4 \bigg(\frac{p-2}{p}\bigg)^2 \bigg] \xi \alpha^2 \|v\|_2^2.
\end{align*}
Choosing here $\gamma = \kappa/\delta_a$, $\kappa = \frac{1}{2}$ we obtain
\[
-\frac{1}{p} \frac{d}{d t} \|v\|^2_2 \geq \bigg[\omega - \frac{2}{p p^\prime}\frac{c(\delta_a)}{\delta_a} \bigg] \|v\|_2^2 - \bigg[ 1+ \frac{p p^\prime}{8} \bigg(\delta_a +4\bigg(\frac{p-2}{p}\bigg)^2 \bigg) \bigg]\xi \alpha^2 \|v\|_2^2.
\]
In particular $\frac{d}{d t}\|u_\alpha\|_2 \leq \frac{2+\delta_a}{2} \xi \alpha^2 \|u_\alpha\|_2,$ and so 
\[
 \|u_\alpha(t)\|_2 \leq  e^{\frac{2+\delta_a}{2}  \xi \alpha^2 (t-s)}\|f\|_2 .\\
\tag{$\star$}
\]
Choosing $\gamma = \kappa/\delta_a$, $\kappa = \frac{1}{4}$ we obtain
\[
-\frac{1}{p} \frac{d}{d t} \|v\|_2^2 \geq \frac{2}{p p^\prime}\|A^{1/2} v \|_2^2 + \bigg[\omega- \frac{1}{p p^\prime}\frac{c(\delta_a)}{\delta_a}\bigg]  \|v\|_2^2 - \bigg[ 1+ \frac{p p^\prime}{4} \bigg(\delta_a +4\bigg(\frac{p-2}{p}\bigg)^2 \bigg) \bigg]\xi \alpha^2 \|v\|_2^2.
\]
Let $p \geq 4.$ Then $\bigg[\omega- \frac{1}{p p^\prime}\frac{c(\delta_a)}{\delta_a}\bigg]\geq 0$ and $\bigg[ 1+ \frac{p p^\prime}{4} \bigg(\delta_a +4\bigg(\frac{p-2}{p}\bigg)^2 \bigg) \bigg] \leq p C_{\delta_a}, \; C_{\delta_a} =1+\frac{\delta_a}{4}.$ Therefore
\[
-\frac{d}{d t} \|v\|_2^2 \geq \|A^{1/2} v \|_2^2 - C_{\delta_a} p^2 \xi \alpha^2 \|v\|_2^2 \\
\tag{$\star\star$}
\]
Using the Nash inequality $\|A^{1/2}v\|_2^2 \geq \sigma C_N \|v\|_2^{2+\frac{4}{d}} \|v\|_1^{-\frac{4}{d}},$ we obtain from $(\star\star)$
\[
-2 \frac{d}{d t} \|v\|_2 \geq \sigma C_N \|v\|_2^{1+\frac{4}{d}} \|v\|_1^{-\frac{4}{d}} - C_{\delta_a} p^2 \xi \alpha^2 \|v\|_2, \text{ or }
\]
\[
\frac{d}{d t} \|v\|_2^{-4/d} \geq \frac{2 \sigma C_N}{d} \|v\|_1^{-4 /d} - \frac{2}{d} C_{\delta_a} p^2 \xi \alpha^2 \|v\|_2^{-4/d}.
\]
The last inequality is linear with respect to $w_p = \|v\|_2^{-4/d}.$ Therefore, setting $c_g :=\frac{2 \sigma C_N}{d}$ and
\[
\mu_p(t):= \frac{2}{d} C_{\delta_a} p^2 \xi \alpha^2 (t-s),
\] 
we have
\begin{align*}
w_p(t) & \geq c_g e^{-\mu_p(t)} \int_s^t e^{\mu_p(r)} w_\frac{p}{2} (r)d r \\
& \geq  c_g e^{-\mu_p(t)} \int_s^t e^{\mu_p(r)} (r-s)^q d r \;V_\frac{p}{2}(t),
\end{align*}
where $ q = \frac{p}{2} - 2$ and 
\begin{align*}
V_\frac{p}{2}(t):= & \inf[(r-s)^{-q}w_\frac{p}{2}(r) \mid  s \leq r \leq t ] \\
= & \bigg \{ \sup \bigg[(r-s)^\frac{q d}{2 p} \|u_\alpha(r)\|_{p/2} \mid s\leq r \leq t\bigg] \bigg \}^{-\frac{2p}{d}}.
\end{align*}
Set $\beta = 2 d^{-1}C_{\delta_a}\xi \alpha^2.$ Since $e^{-\mu_p(t)} \int_s^t e^{\mu_p(r)} (r-s)^q d r \geq e^{- \beta p^2 (t-s)} \int_s^t e^{\beta  p^2 (r-s)}(r-s)^q d r$ and
\begin{align*}
\int_s^t e^{\beta  p^2 (r-s)}(r-s)^q d r & = \bigg(\frac{t-s}{\beta p^2} \bigg)^{q+1} \int_0^{\beta  p^2} e^{(t-s) r} r^q d r \\
& \geq \bigg(\frac{t-s}{\beta p^2} \bigg)^{q+1} e^{\beta (p^2 -1)(t-s)} \int_{\beta p^2 (1-p^{-2})}^{\beta p^2} r^q d r \\
& = \frac{(t-s)^\frac{p-2}{2}}{p-2} 2 \big[1-(1-p^{-2})^{p-1} \big] e^{\beta (p^2 -1)(t-s)} \\
& \geq K p^{-2} (t-s)^\frac{p-2}{2} e^{\beta (p^2 -1)(t-s)},
\end{align*}
where $K:= 2 \inf \big \{ p \big[1-(1-p^{-2})^{p-1} \big] \mid p \geq 2 \big \} > 0,$ we obtain
\[
w_p(t) \geq \tilde{c}_g K p^{-2} e^{-\beta (t-s)} (t-s)^\frac{p-2}{2} V_\frac{p}{2}(t),
\]
or, setting $W_p(t) := \sup \big[ (r-s)^\frac{d(p-2)}{4 p} \|u_\alpha (r) \|_p \mid s \leq r \leq t \big],$
\[
W_p(t) \leq (\tilde{c}_g K)^{-\frac{d}{2 p}} p^\frac{d}{p} e^{\frac{C_{\delta_a} \xi \alpha^2}{p} (t-s)} W_{p/2}(t), \;\; p = 2^k, \;k= 1, 2, \dots.
\]
Iterating this inequality, starting with $ k = 2,$ yields $(t-s)^\frac{d}{4} \|u_\alpha (t)\|_\infty \leq C_g  e^{C_{\delta_a}\xi \alpha^2 (t-s)} W_2 (t).$ Finally, taking into account $(\star),$ we arrive at $$\|H^{t,s}_\alpha \|_{2 \to \infty} \leq (t-s)^{-d/4} C_g e^{ C_{\delta_a} \xi \alpha^2 (t-s)}.$$ The same bound holds for $\|H^{t,s}_\alpha \|_{1 \to 2}.$ To see this it is enough to note that, for $H^+ \equiv H^+(b),$ $(H^+_\alpha(b))^* = H^+_{-\alpha} (-b).$ 
\end{proof}

We obtain $e^{-tH^+}(x,y)\leq C e^{\omega t}t^{-\frac{d}{2}}e^{\alpha\cdot(y-x)+c_4\alpha^2t}$, $c_4=C_{\delta_a}\xi$. The proof of 
$(\mbox{UGB}^{h_+})$ is completed upon putting $\alpha=\frac{x-y}{2c_4}$.
\end{proof}

\subsection{Upper bound for $-\nabla \cdot a_{\varepsilon_1} \cdot \nabla + b_\varepsilon \cdot \nabla$}

\label{fbd_7_subsect}

\begin{theorem}
\label{apr_thm2}
In the assumptions of Theorem 2A, 
there exist generic constants  $c_i$ {\rm ($i=3,4,5$)} such that the heat kernel $u(t,x;s,y)=e^{-(t-s)\Lambda_{\varepsilon_1,\varepsilon}}(x,y)$ of $\Lambda_{\varepsilon_1,\varepsilon}=-\nabla \cdot a_{\varepsilon_1} \cdot \nabla + b_\varepsilon \cdot \nabla$ satisfies
\[
u(t,x;s,y) \leq c_3 k_{c_4}(t-s;x-y) e^{c_5 (t-s)}
\]
for all $x, y \in \mathbb{R}^d$ and $0 \leq s < t < \infty$.
\end{theorem}

\begin{proof}
We have $\Lambda = H^+ - E_\varepsilon\mydiv b_+$, so the proof follows from Theorem \ref{aux_ub_thm} and a standard argument using the Duhamel formula and the fact that $E_\varepsilon\mydiv b_+ \in \mathbf{K}^d_\nu$ (Claim \ref{claim2}). 
(If $c(\delta_a)=0$ and $\lambda(\nu)=0$, then we arrive at global in time Gaussian upper bound.)
\end{proof}

\subsection{A posteriori upper bound} We are in position to complete the proof of Theorem 2A.

In Theorem \ref{apr_thm2} we have established the upper bound on the heat kernel of
$$
\Lambda_{\varepsilon_1,\varepsilon}:=-\nabla \cdot a_{\varepsilon_1} \cdot \nabla + b_\varepsilon \cdot \nabla, \quad D(\Lambda_{\varepsilon_1,\varepsilon})=W^{2,2},
$$
where $a_{\varepsilon_1}:=E_{\varepsilon_1}a \in (H_{\sigma,\xi})$, $\varepsilon_1>0$,
with constants independent of $\varepsilon_1$, $\varepsilon$. It remains to pass to the limit $\varepsilon_1 \downarrow 0$ and then $\varepsilon \downarrow 0$. Since $b \in \mathbf{F}_\delta$ with $\delta$ that is assumed to be only finite, we can not appeal to \cite[Theorems 4.2, 4.3]{KiS2} as in the proof of Theorem \ref{second_thm}. Instead, we will use

\begin{proposition}
\label{prop_approx}
In the assumptions of Theorem 2A,
the limit 
$$
s{\mbox-}L^2\mbox{-}\lim_{\varepsilon \downarrow 0} \lim_{\varepsilon_1 \downarrow 0}e^{-t\Lambda_{\varepsilon_1,\varepsilon}} \quad \text{(locally uniformly in $t \geq 0$)}
$$
exists and determines a positivity preserving $L^\infty$-contraction quasi contraction $C_0$ semigroup in $L^2$, say, $e^{-t\Lambda}$.
\end{proposition}
\begin{proof}

Since $b_\varepsilon \in [L^\infty \cap C^\infty]^d$, the limit
$$
e^{-t\Lambda_\varepsilon}=s{\mbox-}L^2\mbox{-} \lim_{\varepsilon_1 \downarrow 0}e^{-t\Lambda_{\varepsilon_1,\varepsilon}} 
$$
exists and determines a quasi contraction $C_0$ semigroup (positivity preserving $L^\infty$-contraction), and
$\Lambda_\varepsilon=A+b_\varepsilon \cdot \nabla$, $D(\Lambda_\varepsilon)=D(A)$.

 Thus, it remains to pass to the limit $\varepsilon \downarrow 0$. It suffices to prove that $e^{-t\Lambda_\varepsilon}f$ converges strongly in $L^2$ for every $0 \leq f \in C_c^\infty$, and then apply a density argument. 
 
In what follows, the constant $\nu$ is from Theorem \ref{apr_thm2} but possibly taken smaller, if needed, so that $\nu<2\sigma$. We have
\begin{align}
\label{div_b_eq}
\tag{$\star$}
{\rm div\,}b_\varepsilon=E_\varepsilon {\rm div\,}b=E_\varepsilon ({\rm div\,}b_+ - {\rm div\,}b_-) = E_\varepsilon {\rm div\,}b_+ - E_\varepsilon {\rm div\,}b_-,
\end{align}
where $0\leq E_\varepsilon {\rm div\,}b_-\in L^\infty\cap C^\infty$ by assumption (2) of Theorem 2A, and   $0\leq E_\varepsilon {\rm div\,}b_+ \in L^\infty\cap C^\infty$ by Claim \ref{claim2}.

1. Set $u \equiv u_\varepsilon=e^{-t\Lambda_\varepsilon}f$. Using the equation for $u$, we have
$$
\frac{1}{2}\frac{d}{dt}\langle u^2\rangle + \langle a \cdot \nabla u,\nabla u \rangle + \langle b_\varepsilon \cdot \nabla u, u \rangle =0.
$$
Since $u$ satisfies a \textit{qualitative} Gaussian upper bound (i.e.\,with constants that a priori depend on the smoothness of the coefficients), we find that
\begin{align*}
-\langle b_\varepsilon \cdot \nabla u, u \rangle & = \frac{1}{2} \langle {\rm div\,}b_{\varepsilon},u^2\rangle \\
& (\text{we are using \eqref{div_b_eq}}) \\
& \leq \frac{1}{2} \langle e^{\varepsilon \Delta}{\rm div\,}b_+,u^2\rangle.
\end{align*}
Since $E_\varepsilon{\rm div\,} b_+ \in \mathbf{K}^d_\nu$ by Claim \ref{claim2}, $E_\varepsilon{\rm div\,} b_+$ is form-bounded: $\langle E_\varepsilon{\rm div\,}b_{+},u^2\rangle \leq \nu \langle |\nabla u|^2 \rangle + c_\nu \langle u^2\rangle$, $c_\nu=\lambda \nu$  (see the introduction).
Hence
$$
\frac{d}{dt}\langle u^2\rangle + 2\langle a \cdot \nabla u,\nabla u \rangle - \nu\langle |\nabla u|^2\rangle -c_\nu\langle u^2\rangle \leq 0.
$$
Thus, for $t \in [0,T]$,
$$
e^{-c_\nu t}\langle u^2(t)\rangle + (2\sigma-\nu)\int_0^t e^{-c_\nu \tau} \|\nabla u\|_2^2 d\tau \leq \|f\|_2^2,
$$
so
$$
\sup_{\tau \in [0,T]}\langle u^2(\tau)\rangle + c\int_0^T  \|\nabla u\|_2^2 d\tau \leq e^{c_\nu T}\|f\|_2^2
$$
for positive constant $c:=2\sigma-\nu$.

\smallskip

2.~Fix some $\varepsilon_n \downarrow 0$ and put $g=u_{\varepsilon_n}-u_{\varepsilon_m}$. Then, subtracting the equations for $u_{\varepsilon_n}$, $u_{\varepsilon_m}$ arguing as above, multiplying by $g$ and integrating, we obtain 
$$
\sup_{t \in [0,T]}\|g(t)\|_2^2 + c\int_0^T \|\nabla g\|_2^2 d\tau \leq e^{c_\nu T}\int_0^T |\langle (b_{\varepsilon_n}-b_{\varepsilon_m}) \cdot \nabla u_{\varepsilon_m}, g \rangle|d\tau,
$$
where we estimate the RHS as
\begin{equation}
\label{g0}
\tag{$\ast$}
\int_0^T |\langle (b_{\varepsilon_n}-b_{\varepsilon_m}) \cdot \nabla u_{\varepsilon_m}, g \rangle|d\tau \leq \biggl( \int_0^T \|(b_{\varepsilon_n}-b_{\varepsilon_m})g \|_2^2 d\tau \biggr)^{\frac{1}{2}}\,\biggl( \int_0^T \|\nabla u_{\varepsilon_m} \|_2^2 d\tau \biggr)^{\frac{1}{2}}.
\end{equation}
By Step 1, the second multiple in the RHS of \eqref{g0} is uniformly (in $m$) bounded. To estimate the first multiple, we can appeal to the a priori Gaussian upper bound on the heat kernel of $\Lambda_{\varepsilon_1,\varepsilon}$ (Theorem \ref{apr_thm2}) to obtain pointwise estimate 
\begin{equation}
\label{g}
\tag{$\ast\ast$}
|g(t,x)| \leq 2 \hat{c}_3 \langle k_{c_4}(t,x - \cdot)f(\cdot)\rangle \quad (=:F(t,x))
\end{equation}
on $[0,T] \times \mathbb R^d$.
We write ($R>0$)
$$
\int_0^T \|(b_{\varepsilon_n}-b_{\varepsilon_m})g \|_2^2 d\tau \leq \int_0^T \|\eta_R(b_{\varepsilon_n}-b_{\varepsilon_m})g \|_2^2 d\tau + \int_0^T \|(1-\eta_R)(b_{\varepsilon_n}-b_{\varepsilon_m})g \|_2^2 d\tau,
$$
where $0 \leq \eta_R \in C_c^\infty$, $0 \leq \eta \leq 1$, $\eta_R \equiv 1$ on $B(0,R)$. Now, for every $R>0$, the first term converges to $0$ as $n$, $m \rightarrow \infty$ since $|b_{\varepsilon_n}-b_{\varepsilon_m}| \rightarrow 0$ in $L^2_{\loc}$ and, by \eqref{g}, $g$ is uniformly in $n$, $m$ bounded on $[0,T] \times \mathbb R^d$. 
In turn, the second term is estimated using \eqref{g} and  $b_\varepsilon \in \mathbf{F}_\delta$:
\begin{align*}
\|(1-\eta_R)(b_{\varepsilon_n}-b_{\varepsilon_m})g(\tau) \|_2^2 & \leq \|(1-\eta_R)(b_{\varepsilon_n}-b_{\varepsilon_m})F(\tau) \|_2^2 \\
& \leq 2\delta \big\|\nabla [(1-\eta_R)F(\tau)]\big\|_2^2 + 2c(\delta) \|(1-\eta_R) F(\tau)\|_2^2, \quad \tau \in [0,T].
\end{align*}
Taking into account that $f \in C_c^\infty$, it is easily seen that the last expression  can be made as small as needed, uniformly in $\tau$, by selecting $R$ sufficiently large. 

It follows that the first multiple in \eqref{g0} tends to $0$ as $n,m \rightarrow \infty$.

Thus, $\{u_{\varepsilon_n} \equiv e^{-t\Lambda_{\varepsilon_n}}f\}_{n=1}^\infty$ is a Cauchy sequence in $L^\infty([0,T],L^2(\mathbb R^d))$. We set
$$
U^tf:=s\mbox{-}L^2\mbox{-}\lim_{\varepsilon_n \downarrow 0} e^{-t\Lambda_{\varepsilon_n}}f, \quad 0<t \leq T. 
$$
Next, we extend $U^t$, $0<t \leq T$ by continuity to whole $L^2$, and then, using the reproduction property of $e^{-t\Lambda_{\varepsilon_n}}$, extend it to all $0<t<\infty$. The strong continuity of $U^t$ and the other claimed properties now follow from the corresponding properties of $e^{-t\Lambda_{\varepsilon_n}}$. Set $e^{-t\Lambda}:=U^t$.

The proof of Proposition \ref{prop_approx} is completed.
\end{proof}

\begin{remark*}
The proof of Proposition \ref{prop_approx} can be made independent of Theorem \ref{apr_thm2} by working with appropriate weights, essentially repeating the proof of \cite[Theorem 4.3]{KiS2}.
\end{remark*}

\noindent\textit{Proof of Theorem 2A.} Theorem \ref{apr_thm2} and Proposition \ref{prop_approx} yield
$$
\|e^{-t\Lambda}\|_{1 \rightarrow \infty} \leq c_3e^{c_5 t}t^{-\frac{d}{2}}, \quad t>0.
$$
Hence, by the Dunford-Pettis Theorem, $e^{-t\Lambda}$ is an integral operator for every $t>0$.

Next, for every pair of balls $B_1$, $B_2 \subset \mathbb R^d$ we have, using again  Theorem \ref{apr_thm2} and Proposition \ref{prop_approx}:
$$
\langle \mathbf{1}_{B_1},e^{-t\Lambda}\mathbf{1}_{B_2} \rangle \leq c_3 e^{c_5 t} \langle \mathbf{1}_{B_1},e^{tc_4\Delta}\mathbf{1}_{B_2} \rangle.
$$
Since for every $t>0$ $e^{-t\Lambda}$ is an integral operator, the a posteriori Gaussian upper bound in Theorem 2A follows by applying the Lebesgue Differentiation Theorem.

\bigskip

\section{Proof of Theorem 2B}

\label{proof_prime}

Since $b \in \mathbf{MF}_\delta$, the vector fields $b_\varepsilon=e^{\varepsilon\Delta}b$ are $C^\infty$ smooth are bounded (following the proof of Claim \ref{claim1} in Section \ref{apost_sect0}) and are in class $\mathbf{MF}_\delta$ with the same constants $\delta$ and $c(\delta)$. Indeed, 
\begin{align*}
&\langle b_\varepsilon f,f\rangle=\langle b e^{\varepsilon\Delta}|f|^2\rangle=\langle b (h_\varepsilon)^2\rangle,\\
&\text{ where } h_\varepsilon=\sqrt{E_\varepsilon |f|^2} \text{ so } \nabla h_\varepsilon=h_\varepsilon^{-1}E_\varepsilon(|f|\nabla|f|,\\
&\|\nabla h_\varepsilon\|_2^2\leq \|\sqrt{E_\varepsilon(\nabla|f|)^2}\|_2^2=\|E_\varepsilon(\nabla|f|)^2\|_1\leq \|\nabla f\|_2^2
\end{align*}
and $\|h_\varepsilon\|_2 \leq \|f\|_2$, which clearly yields the required. 

Thus, in what follows, we assume that $b \equiv b_\varepsilon \in \mathbf{MF}_\delta$ is bounded and $C^\infty$ smooth. The assumption (2) of Theorem 2B ensures that ${\rm div\,}b_- \equiv E_\varepsilon {\rm div\,}b_- \in L^\infty \cap C^\infty$. Further, assumption (3) and Claim \ref{claim2} in Section \ref{kato_rem_sect} ensure that ${\rm div\,}b_+  \equiv E_\varepsilon {\rm div\,}b_+ \in \mathbf{K}^d_\nu$ with the same constants $\nu$ and $\lambda$.

The rest of the proof follows closely Sections \ref{fbd_6_subsect} and \ref{fbd_7_subsect} of the proof of Theorem 2A with the following modification. 
We need to estimate differently the term $\langle\alpha \cdot b, v^2 \rangle$. By $b \in \mathbf{MF}_\delta$,
\begin{align*}
|\langle\alpha \cdot b, v^2 \rangle| & = |\alpha \cdot \big \langle b v, v \big \rangle|  \leq |\alpha| |\langle bv,v\rangle| \\
& \leq |\alpha|\big( \delta \sigma^{-\frac{1}{2}}\|A^{1/2}v\|_2 \|v\|_2 +c(\delta)^{\frac{1}{2}}\|v\|^2_2 \big) \\
&\leq \frac{4 \gamma \delta}{p p^\prime} \|A^{1/2} v\|_2^2 + \bigg( c(\delta) + \alpha^2 \bigg(\frac{1}{4}+\frac{\delta}{4 \sigma\gamma}\frac{p p^\prime}{4}\bigg) \bigg)\|v\|^2_2, 
\end{align*}
and so
\begin{align*}
-\frac{1}{p} \frac{d}{d t}\|v\|_2^2 & \geq \frac{4}{p p^\prime}(1-\kappa -\gamma \delta) \|A^{1/2} v \|_2^2  -[\omega-c(\delta)] \|v\|_2^2 \\
& - \bigg[\frac{1}{4}+\xi +\frac{\delta}{4 \sigma\gamma}\frac{p p^\prime}{4} + \frac{1}{4 \kappa}\frac{p p^\prime}{4} 4 \bigg(\frac{p-2}{p}\bigg)^2\xi \bigg] \alpha^2 \|v\|_2^2.
\end{align*}
Take $\omega=c(\delta)$.
Choosing first $\gamma =\frac{\kappa}{\delta}$, $\kappa=\frac{1}{2}$, $p=2$, and then $\gamma =\frac{\kappa}{\delta}$, $\kappa=\frac{1}{4}, \;p \geq 4,$ we have
\[
\|u_\alpha(t) \|_2 \leq \|f\|_2 \exp\bigg[\frac{2(\frac{1}{4}+\xi)\sigma+\delta^2}{2\sigma}  \alpha^2 (t-s)\bigg]
\tag{$\star^a$}
\]
and
\begin{align*}
-\frac{1}{p} \frac{d}{d t}\|v\|_2^2 & \geq \frac{2}{p p^\prime} \|A^{1/2} v \|_2^2  
- \bigg[\frac{1}{4} + \xi + \frac{\delta^2}{\sigma} \frac{p p^\prime}{4} + \frac{p p^\prime}{4} 4 \bigg(\frac{p-2}{p}\bigg)^2\xi \bigg] \alpha^2 \|v\|_2^2 \\
& \geq \frac{2}{p p^\prime} \|A^{1/2} v \|_2^2  
- C_{\delta,\sigma,\xi} p \alpha^2 \|v\|_2^2.
\end{align*}
Therefore, 
\begin{align*}
 -\frac{d}{d t}\|v\|_2^2 & \geq \|A^{1/2}v\|_2^2 - C_{\delta,\xi,\sigma} p^2 \alpha^2 \|v \|_2^2 \\
& \geq \sigma C_N \|v\|_2^{2+\frac{4}{d}} \|v\|_1^{-\frac{4}{d}} - C_{\delta,\sigma,\xi} p^2 \alpha^2 \|v \|_2^2,
\end{align*}
so
\begin{align*}
\frac{d}{d t} \|v\|_2^{-4/d} & \geq \frac{2\sigma C_N}{d} \|v\|_1^{-\frac{4}{d}} - \frac{2}{d} C_{\delta,\sigma,\xi} p^2\alpha^2\|v \|_2^{-4/d}.
 \end{align*}
Now we iterate the last inequality in the same way as in the proof of Theorem 2A, arriving at
$$(t-s)^\frac{d}{4} \|u_\alpha (t)\|_\infty \leq C_g  e^{C_{\delta,\sigma,\xi} \alpha^2 (t-s)} W_2 (t),$$ 
where $W_p(t) := \sup \big[ (r-s)^\frac{d(p-2)}{4 p} \|u_\alpha (r) \|_p \mid s \leq r \leq t \big]$.
Taking into account $(\star^a),$ we arrive at $\|H^{t,s}_\alpha \|_{2 \to \infty} \leq (t-s)^{-d/4} C'_g e^{C'_{\delta,\sigma,\xi} \alpha^2 (t-s)}.$ The same bound holds for $\|H^{t,s}_\alpha \|_{1 \to 2}.$ To see this it is enough to note that, for $H^+ \equiv H^+(b),$ $(H^+_\alpha(b))^* = H^+_{-\alpha} (-b).$

\medskip

We obtain $e^{-tH^+}(x,y)\leq C e^{\omega t}t^{-\frac{d}{2}}e^{\alpha\cdot(y-x)+c_4\alpha^2t}$, $c_4=C_{\delta,\sigma,\xi}'$. Putting $\alpha=\frac{x-y}{2c_4}$, we obtain 
$(\mbox{UGB}^{h_+})$.
Now argue as in Section \ref{fbd_7_subsect}.

\bigskip

\section{Proof of Theorem 3A}

\label{sect_lb}

In the assumptions of Theorem 3A the upper bound of Theorem 2A is valid, so we only need to prove the lower bound.

We will prove the lower bound in Theorem 3A first for the smoothed out coefficients $a_{\varepsilon_1}$, $b_\varepsilon$ (Theorem \ref{apr_thm3} below). Recall that $b_\varepsilon$ are bounded and are in $\mathbf{F}_\delta$ with the same $c(\delta)$ (thus, independent of $\varepsilon$), see Claim \ref{claim1}. 

First, we assume $0 < t-s \leq 1$.

Write $$A_{\varepsilon_1}=-\nabla\cdot a_{\varepsilon_1}\cdot\nabla, \quad \Lambda_{\varepsilon_1,\varepsilon}=A_{\varepsilon_1}+b_\varepsilon\cdot\nabla, \quad \mydiv b_\varepsilon=E_\varepsilon \mydiv b_+-E_\varepsilon \mydiv b_-.$$ 

We have ${\rm div\,}b_\varepsilon \in \mathbf{K}^d_\nu$ with the same constants $\nu$, $\lambda(\nu)$ (see the beginning of the proof of Theorem 2A for details).

By Theorem \ref{apr_thm2}, the heat kernel $u(t,x;s,y)$ of $\Lambda_{\varepsilon_1,\varepsilon}$ satisfies, for all $x,y \in \mathbb R^d$, the Gaussian upper bound
$$
u(t,x;s,y) \leq \hat{c}_3 k_{c_4}(t-s;x-y), \quad 0 < t-s \leq 1,
$$
for generic constants $\hat{c}_3$, $c_4$.
The latter trivially yields the integral bound
\begin{equation}
\label{int_bd_two_star}
\sup_{x \in \mathbb{R}^d} \langle u^{2}(t,x;s,\cdot)\rangle \leq \hat{c} (t-s)^{-\frac{d}{2}}, \quad 0 < t-s \leq 1
\tag{$\circ\circ$}
\end{equation}
with generic $\hat{c}$. We will use this integral bound below.

\medskip

\subsection{$\hat{G}$-bound} \label{fbd_3_subsect}

Let us define Nash's function
\[
\hat{G}(s) 
:= \langle k_\beta(t-s,o-\cdot) \log u(t,x;s,\cdot) \rangle, \quad o = \frac{x+y}{2}
\]
for all $0 <t-s \leq 1$ and $x,y \in \mathbb{R}^d$ such that $2 |x-y| \leq \sqrt{\beta(t-s)}.$ 

\begin{proposition}
\label{g1_thm}
There exist generic constants $\beta$ and $\mathbb{C}$ such that 
$$
\hat{G}(t_s) \geq - \tilde{Q}(t-t_s) - \mathbb{C}, \qquad t_s=\frac{t+s}{2},
$$
where, recall, $\tilde{Q}(t-t_s) := \frac{d}{2} \log (t-t_s).$
\end{proposition}

\begin{proof}
In the proof of Proposition \ref{prop_G2} take $r=2$ and instead of \eqref{int_bd_one_star} from Section \ref{int_bd_sect} use integral bound \eqref{int_bd_two_star}.
\end{proof}

\medskip

\subsection{$G$-bound for $-\nabla \cdot a_{\varepsilon_1} \cdot \nabla + \nabla \cdot b_\varepsilon$} \label{fbd_4_subsect}

Let $u_*(t,x;s,y)$ denote the heat kernel of $\Lambda_*=A_{\varepsilon_1} + \nabla \cdot b_\varepsilon$.
Set
\[
G(t)
:= \langle k_\beta(t-s,o -\cdot)\log u_*(t, \cdot; s,y) \rangle,
\]
where  $0 \leq s < t < \infty$ and $ x,y \in \mathbb{R}^d$ such that $2 |x-y| \leq \sqrt{\beta(t-s)}.$

\begin{proposition}
\label{g2_thm}
Let $\beta$ and $\mathbb{C}$ be (generic) constants defined in Proposition \ref{g1_thm} and Proposition \ref{mb_thm}, respectively. Then
\[
G(t)
\geq - \tilde{Q}(t-s) - \mathbb{C}.
\]
\end{proposition}

\begin{proof}
We repeat the proof of Proposition \ref{prop_G3} with $p=2$.
\end{proof}

\subsection{Lower bound for the auxiliary operator $-\nabla \cdot a_{\varepsilon_1} \cdot \nabla + b_\varepsilon \cdot \nabla - E_\varepsilon{\rm div\,}b_-$.}

\label{fbd_5_subsect}

Set $$H^- := \Lambda_{\varepsilon_1,\varepsilon} - E_\varepsilon\mydiv b_-.$$ Let $H^{t,s}f$ denote the solution of
\[ 
	\left\{ \begin{array}{rcl}
		-\frac{d}{d t} H^{t,s} f = H^- H^{t,s} f & , & 0 < t-s \leq 1, \\
	0 \leq f \in L^1 \cap L^\infty. & 
	\end{array} \right.
	\tag{$CP_{H^-}$}
\]
Let $h(t):= H^{t,s}f.$ \textit{It is seen (for example from the Duhamel formula) that $u(t,x;t_s,y) \leq h(t,x;t_s,y)$ and $u_*(t_s,x;s,y) \leq h(t_s,x;s,y),$ where $u(t), u_*(t)$ solve $(C P_\Lambda), (C P_{\Lambda_*})$ respectively}. It is seen that
\[
h(t,x;s,y) \geq (4 \pi \beta (t-t_s))^{d/2} \langle k_\beta (t-t_s,o -\cdot) h(t,x;t_s,\cdot) h(t_s,\cdot;s,y) \rangle,
\]
\[
k_\beta (t-t_s,o -\cdot)=k_\beta (t_s-s,o -\cdot),
\]
and, for all $2 |x-y| \leq \sqrt{\beta(t-t_s)}\;,$ due to Proposition \ref{g1_thm} and Proposition \ref{g2_thm},
\begin{align*}
\log h(t,x;s,y) & \geq \log (4 \pi \beta)^{d/2} + \tilde{Q}(t-t_s)\\
& + \langle k_\beta (t-t_s,o-\cdot) \log u(t,x;t_s,\cdot) \rangle + \langle k_\beta (t-t_s,o-\cdot) \log u_*(t_s,\cdot; s,y) \rangle \\
& \geq \log (4 \pi \beta)^{d/2} - \tilde{Q}(t-t_s)  -2 \mathbb{C} \\
& = - \tilde{Q}(t-s)  -2 \mathbb{C} + \log (8 \pi \beta)^{d/2},
\end{align*}
i.e.\,we have proved a lower Gaussian bound for $h(t,x;s,y)$ but only for  $2 |x-y| \leq \sqrt{\beta(t-t_s)}$.
Now, the standard argument (see e.g.\,\cite[Theorem 3.3.4]{Da}) gives

\begin{theorem}
\label{lb_div}
There exist generic constants $c_1$, $c_2 > 0$ such that, for all $x,y \in \mathbb{R}^d$
\begin{equation}
c_1 k_{c_2}(t-s,x-y) \leq h(t,x;s,y)
\tag{$\mbox{LGB}^{h_-}$}
\end{equation}
for all $0 < t-s\leq 1$.

\end{theorem}

\bigskip

\subsection{Lower bound for $-\nabla \cdot a_{\varepsilon_1} \cdot \nabla + b_\varepsilon \cdot \nabla$.}

\label{fbd_8_subsect}

Let $u(t,x;s,y)$ be the heat kernel of $\Lambda_{\varepsilon_1,\varepsilon}$.

\begin{theorem}
\label{apr_thm3}
There exist generic constants $c_0 \geq 0$ and $c_i>0$ ($i=1,2$) such that, for all $x, y \in \mathbb{R}^d$ and $0 \leq s < t < \infty,$
$$
  c_1 k_{c_2}(t-s;x-y) e^{-c_0 (t-s)} \leq u(t,x;s,y).
$$
\end{theorem}

\begin{proof}

Let $h_1(t,x;s,y)$, $h_{p^\prime}(t,x;s,y)$ denote the heat kernels of $H^-=-\nabla \cdot a_{\varepsilon_1} \cdot \nabla + b_\varepsilon \cdot \nabla-E_\varepsilon\mydiv b_-$,   $H^-_{p'}=-\nabla \cdot a_{\varepsilon_1} \cdot \nabla + b_\varepsilon \cdot \nabla - p^\prime E_\varepsilon\mydiv b_-$, respectively.
The pointwise inequality
\begin{equation}
\label{point_ineq}
\tag{$\star$}
h_1(t,x;s,y) \leq \big[ u(t,x;s,y) \big]^{1/p} \big[ h_{p^\prime}(t,x;s,y) \big]^{1/p^\prime}, \;\; p >1,
\end{equation}
is a standard consequence of the Lie-Trotter Product Formula (for the proof, if needed, see \cite{HS}). 

1.~In the RHS of \eqref{point_ineq}, we bound $h_{p^\prime}(t,x;s,y)$ from above as follows. We write the Duhamel series for $h_{p^\prime}(t,x;s,y)$, with $H^-_{p'}$ viewed as $H_{p'}^+=\Lambda_{\varepsilon_1,\varepsilon} + p'E_\varepsilon{\rm div\,}b_+$ perturbed by $-p'\big(E_\varepsilon\mydiv b_++E_\varepsilon\mydiv b_-\big)$, and estimate its terms from above using a straighforward modification of Theorem \ref{aux_ub_thm} and appealing to $|\mydiv b| \in \mathbf{ K}^{d}_{\nu}$. 
We obtain
\begin{equation*}
h_{p'}(t,x;s,y) \leq \tilde{c}_3 k_{\tilde{c}_4}(t-s;x-y)
\end{equation*}
for all $x, y \in \mathbb{R}^d$ and $0 \leq s < t \leq T$, for generic constants $\tilde{c}_i$ ($i=3,4,5$).

2.~In the LHS of \eqref{point_ineq}, we bound $h_1(t,x;s,y)$ from below using Theorem \ref{lb_div}.  

Now, 1-2 yield the required lower bound on $u(t,x;s,y)$ for $x$, $y \in \mathbb R^d$, $0 < t-s \leq 1$.
Next, the reproduction property of $u(t,x;s,y)$ gives the required lower bound for all $x, y \in \mathbb{R}^d$ and $0 \leq s< t <\infty$. 

If $c(\delta_a)=\lambda(\nu)=0$, then we work over $0 \leq t < \infty$ from the beginning, obtaining a global in time lower bound.
\end{proof}

\subsection{A posteriori lower bound}

We are in position to prove Theorem 3A.
Theorem \ref{apr_thm3} and Proposition \ref{prop_approx} yield for every pair of balls $B_1$, $B_2 \subset \mathbb R^d$
$$
c_1 e^{-c_0 t} \langle \mathbf{1}_{B_1},e^{tc_2 \Delta}\mathbf{1}_{B_2} \rangle \leq \langle \mathbf{1}_{B_1},e^{-t\Lambda}\mathbf{1}_{B_2} \rangle,
$$
so an application of the Lebesgue Differentiation Theorem gives the a posteriori Gaussian lower bound in Theorem 3A.

\bigskip

\section{Proof of Theorem 3B} 

The upper bound follows from Theorem 2B.
The proof of the lower bound under the assumption ${\rm div}\,b=0$ was given \cite{S2}. Below we relax that assumption to ``${\rm div}\,b \in \mathbf{K}^d_\nu$ for $\nu$ sufficiently small'' by modifying the proof in \cite{S2} and then arguing as in the proof Theorem 3A.

Few remarks are in order. 

For every $\varepsilon>0$, $b_\varepsilon \in \mathbf{MF}_{\delta}$ with the same constants $\delta$, $c(\delta)$, and ${\rm div\,}b_\varepsilon \in \mathbf{K}^d_\nu$ with the same constants $\nu$, $\lambda(\nu)$ (for details, see the beginning of the proof of Theorem 2B and of Theorem 2A, respectively). In particular,
\begin{equation}
\label{divb_fbd}
\tag{$\star$}
\||{\rm div\,}b_\varepsilon|^{\frac{1}{2}}f\|_2^2 \leq \nu \|\nabla f\|_2^2 + \lambda \nu \|f\|_2^2, \quad f \in W^{1,2}.
\end{equation}

In what follows, we put $$b \equiv b_\varepsilon, \quad {\rm div\,}b \equiv {\rm div\,}b_\varepsilon =E_\varepsilon {\rm div\,}b_+ - E_\varepsilon {\rm div\,}b_-.$$
We denote $ E_\varepsilon {\rm div\,}b_\pm$, with some abuse of notation, by ${\rm div\,}b_\pm$.

We will establish the lower bound for $0 < t-s \leq 1$. Then the reproduction property will yield the lowe bound for all $0 < t-s <\infty$.

\subsection{$\hat{G}$-bound for $-\nabla \cdot a \cdot \nabla + b \cdot \nabla$} \label{fbd_44_subsect}

By Theorem 2B, the heat kernel $u(t,x;s,y)$ of $\Lambda_{\varepsilon_1,\varepsilon}=-\nabla \cdot a_{\varepsilon_1} \cdot \nabla + b_\varepsilon \cdot \nabla$ satisfies, for all $x,y \in \mathbb R^d$, the Gaussian upper bound
\begin{equation}
\label{ub_11}
\tag{${\rm UGB}_u$}
u(t,x;s,y) \leq \hat{c}_3 k_{c_4}(t-s;x-y), \quad 0 < t-s \leq 1,
\end{equation}
for generic constants $\hat{c}_3$, $c_4$.

The next proposition is valid under weaker assumptions than those of Theorem 3B, namely, it suffices to assume that \eqref{ub_11} holds, and 
$$
\|({\rm div\,}b_-)^{\frac{1}{2}}f\|_2^2 \leq \nu \|\nabla f\|_2^2 + \lambda \nu \|f\|_2^2, \quad f \in W^{1,2}
$$
with e.g.\,$\nu \leq \frac{\sigma}{8}$.

\begin{proposition}
\label{g1_thm_3}
Let $x,y \in \mathbb R^d$, $o=\frac{x+y}{2}$, $t_s=\frac{t+s}{2}$.
There exist generic constants $\beta$ and $\mathbb{C}$ such that 
\begin{align*}
\hat{G}(t_s) 
 :=\langle k_\beta(t-t_s,o-\cdot) \log u(t,z;t_s,\cdot) \rangle \geq - \tilde{Q}(t-t_s) - \mathbb{C}, \quad \text{ for all }z \in B(o,\sqrt{t-t_s}),
\end{align*}
where, recall, $\tilde{Q}(t-t_s) := \frac{d}{2} \log (t-t_s).$
\end{proposition}

\begin{proof}[Proof of Proposition \ref{g1_thm_3}]

Fix $\epsilon>0$ and define
\[ 
G_\epsilon(\tau):=\langle k_\beta(t-t_s,o-\cdot) \log \big[\epsilon k_\beta(t-t_s,o-\cdot)  +u(t,z;\tau,\cdot) \big] \rangle,
\]
where $\tau \in [t_s,\frac{t+t_s}{2}]$.
Then
\[
\hat{G}(t_s)=\inf_{\epsilon>0} G_\epsilon (t_s).
\]


Below we write for brevity:
\[
G_\epsilon(\tau) \equiv \big \langle \Gamma \log \big[ \epsilon\Gamma + U \big] \big \rangle \equiv \big \langle \Gamma \log \big[ \epsilon \Gamma + U(\tau) \big] \big \rangle,
\]
where $\Gamma \equiv \Gamma_\beta \equiv k_\beta(t-t_s,o-\cdot), \; U \equiv U(\tau) \equiv u(t,z;\tau,\cdot).$

Also, set
\[
V := c_0 (t-t_s)^{d/2} \big[\epsilon \Gamma+ U \big], \; c_0 = (4 \pi c_4)^{d/2} e^{-1}\big[\epsilon + \hat{c}_3 e^{\frac{1}{4c_4}}\big]^{-1}.
\]
If $\beta \geq 2c_4,$ then clearly  
\[
V(\tau,y) \exp \frac{|o-y|^2}{4 \beta (t-t_s)} \leq e^{-1} \text{ for all } y \in \mathbb{R}^d, \; \epsilon \in ]0, 1] \text{ and } \tau \in \big [t_s,\frac{t+t_s}{2} \big].
\]
In particular, $-\log V \geq 1$.

Let us calculate $- \partial_\tau G_\epsilon(\tau).$ We have
\begin{align*}
- \partial_\tau G_\epsilon(\tau) & = \bigg \langle \Gamma \frac{-\partial_\tau U}{\epsilon\Gamma + U} \bigg \rangle = \bigg \langle \frac{\Gamma}{\epsilon \Gamma + U} (\nabla \cdot a \cdot \nabla + \nabla \cdot b)U \bigg \rangle \\
& = \bigg \langle \nabla \log V \cdot a \Gamma \cdot \frac{\nabla U}{\epsilon \Gamma+ U} \bigg \rangle - \bigg \langle \nabla \Gamma \cdot a \cdot \frac{\nabla U}{\epsilon\Gamma + U}  \bigg \rangle + \bigg \langle \Gamma \frac{b \cdot\nabla U}{\epsilon\Gamma + U} \bigg \rangle + \bigg \langle \Gamma \frac{{\rm div}b\;U}{\epsilon\Gamma + U} \bigg \rangle \\
& = \big \langle \nabla \log V \cdot a \Gamma \cdot \nabla \log V \big \rangle -\bigg \langle \nabla \log V \cdot a \Gamma \cdot \frac{\epsilon\nabla \Gamma}{\epsilon \Gamma+ U} \bigg \rangle - \big \langle \nabla \Gamma \cdot a \cdot \nabla \log V \big \rangle \\
&+\bigg \langle \nabla \Gamma \cdot a \cdot \frac{\epsilon\nabla \Gamma}{\epsilon\Gamma + U}  \bigg \rangle + \big \langle \Gamma b \cdot \nabla \log V \big \rangle - \bigg \langle \Gamma \frac{b \cdot\epsilon\nabla \Gamma}{\epsilon\Gamma + U} \bigg \rangle + \bigg \langle \Gamma \frac{U{\rm div}b}{\epsilon\Gamma + U} \bigg \rangle.
\end{align*}
All the terms except for $\big \langle \Gamma \frac{U{\rm div}b}{\epsilon\Gamma + U} \big \rangle$ will be treated as in \cite{S2}. 
Setting $\mathcal{N}:= \big \langle \nabla \log V \cdot a \Gamma \cdot \nabla \log V \big \rangle$, applying quadratic inequality and estimating $\big \langle \Gamma \frac{U{\rm div}b}{\epsilon\Gamma + U} \big \rangle \geq -\big \langle \Gamma{\rm div}b_-\big \rangle$, we have
\begin{align*}
- \partial_\tau G_\epsilon(\tau) & \geq \mathcal{N} - 2 \mathcal{N}^{1/2} \bigg \langle \nabla \Gamma \cdot \frac{a}{\Gamma} \cdot \nabla \Gamma \bigg \rangle^{1/2} + \big \langle \Gamma b \cdot \nabla \log V \big \rangle-\langle |b\cdot\nabla\Gamma|\rangle-\big \langle \Gamma{\rm div}b_-\big \rangle
\end{align*}

\begin{remark}
Note that now we cannot estimate the term $\big \langle \Gamma b \cdot \nabla \log V \big \rangle$ as in the proof of Theorem \ref{second_thm} or Theorem 3A since for any $p>1$ (close to $1$) there are $b\in \mathbf{MF}_\delta$ with $|b|\notin L^p_\loc$.
\end{remark}

Hence
\begin{align*}
- \partial_\tau G_\epsilon(\tau)  & \geq (1-\gamma)\mathcal{N}-\frac{\xi}{\gamma}\bigg\langle\frac{(\nabla\Gamma)^2}{\Gamma}\bigg\rangle + \big \langle b\cdot\nabla\Gamma, - \log V \big \rangle  \\
& - \langle \Gamma {\rm div} b_-,-\log V\rangle -\langle |b||\nabla\Gamma|\rangle - \big \langle \Gamma {\rm div}b_- \big \rangle,
\end{align*}
where $0<\gamma<1$ will be chosen later.

We have:
$$
\big \langle \frac{(\nabla \Gamma)^2}{\Gamma} \big \rangle = \frac{d}{2 \beta} \frac{1}{t-t_s}, \qquad \langle |b||\nabla\Gamma|\rangle\leq \frac{\sqrt{2}}{\sqrt{\beta (t-t_s)}}\langle |b|\Gamma_{2\beta}\rangle.
$$
Further, applying $b \in \mathbf{MF}_\delta$ and \eqref{divb_fbd}, we estimate
\begin{align*}
\langle |b|\Gamma_{2\beta}\rangle & \leq \delta\sqrt{\|\nabla\sqrt{\Gamma_{2\beta}}\|^2_2 + c(\delta)\|\sqrt{\Gamma_{2\beta}}\|^2_2}\|\sqrt{\Gamma_{2\beta}}\|_2 \\
& =\delta \|\nabla\sqrt{\Gamma_{2\beta}}\|^2_2\|\sqrt{\Gamma_{2\beta}}\|_2+\sqrt{c(\delta)}\|\sqrt{\Gamma_{2\beta}}\|^2_2 \\
&  =\frac{\delta}{4}\frac{\sqrt{d}}{\sqrt{\beta(t-t_s)}} +\sqrt{c(\delta)} \\
& \leq \frac{\sqrt{d}\delta + \sqrt{2}\sqrt{\beta}\sqrt{c(\delta)}}{4\sqrt{\beta(t-t_s)}} \quad (\text{we used $0<t-t_s \leq \frac{1}{2}$}),
\end{align*}
\begin{align*}
\langle \Gamma {\rm div\,}b_-  \rangle & \leq \frac{\nu}{4}\frac{d}{2 \beta} \frac{1}{t-t_s} + \lambda\nu  \leq \frac{\frac{d\nu}{2} + 2\beta \lambda \nu  }{4 \beta(t-t_s)},
\end{align*}
\begin{align*}
|\big \langle b \cdot \nabla \Gamma , -\log V \big \rangle| & + \langle \Gamma {\rm div}b_-, -\log V \rangle \\
& \leq \frac{1}{\sqrt{\beta (t-t_s)}} \big \langle |b| \Gamma, -\log V \big \rangle^{1/2} \bigg \langle |b| \frac{|o-\cdot|^2}{4 \beta (t-t_s)} \Gamma, -\log V \bigg \rangle^{1/2} + \langle \Gamma {\rm div}b_-, -\log V \rangle\\
& \equiv \frac{1}{\sqrt{\beta (t-t_s)}} A_0^{1/2} A_2^{1/2} + A_3 \leq \frac{1}{2 \sqrt{\beta (t-t_s)}}(A_0+A_2)+A_3,
\end{align*}
where
$$
A_0(\tau):= \big \langle |b| \Gamma (-\log V) \big \rangle=:\langle |b|,\varphi \rangle,
$$
$$
A_2(\tau):=\bigg \langle |b| \frac{|o-\cdot|^2}{4 \beta (t-t_s)} \Gamma (-\log V) \bigg \rangle=:\langle |b|,\psi\rangle,
$$
$$
A_3(\tau):=\langle \Gamma {\rm div}b_-, -\log V \rangle.
$$
Denoting $$Y(\tau):= G_\epsilon(\tau) +\tilde{Q}(t-\tau)$$ and gathering the above estimates, we obtain 
\begin{align*}
-\partial_\tau Y(\tau)  \geq (1-\gamma)\mathcal{N} -\frac{K_0}{4\beta(t-t_s)}    - \frac{1}{2 \sqrt{\beta (t-t_s)}} \big [ A_0(\tau) +A_2(\tau) \big ]-A_3(\tau),
\end{align*}
where $K_0:=\sqrt{2d}\delta + 2\sqrt{\beta}\sqrt{c(\delta)}+\frac{2\xi d}{\gamma} + \frac{d\nu }{2} + 2\beta \lambda \nu$. 
Multiplying this inequality by $e^{\mu(\tau)},$ $$\mu(\tau):= - \frac{K(t-\tau)}{\beta (t-t_s)},$$ where constant $K$ will be chosen later, we obtain
\begin{align*}
- \partial_\tau \big( e^{\mu(\tau)} Y(\tau) \big) \geq e^{\mu(\tau)} & \bigg [ (1-\gamma)\mathcal{N}(\tau) -  Y(\tau) \partial_\tau \mu(\tau) -\frac{K_0}{4\beta(t-t_s)}\\
& - \frac{1}{2\sqrt{\beta (t-t_s)}}\big[A_0(\tau) +A_2(\tau) \big] - A_3(\tau)\bigg ].
\end{align*}
We note that $$Y(\tau) < c, \quad \tau \in [t_s,(t+t_s)/2],$$ 
where the constant $c=\log (1+\hat{c}_3)$ with $\hat{c}_3$ from  $u(t,x;\tau,\cdot) \leq \hat{c}_3 (t-\tau)^{-d/2}.$ 
Indeed, for $\epsilon\leq (4\pi \beta)^\frac{d}{2}$,
\[
G_\epsilon(\tau)= \big \langle \Gamma \log (\epsilon \Gamma + U) \big \rangle \leq \big \langle \Gamma \big \rangle \log \big[(1+\tilde{c}) (t-\tau)^{-/2}] < -\tilde{Q}(t-\tau) + \log (1+\tilde{c}).
\]
Thus, avoiding division on possible zero, we obtain 
\begin{equation}
\partial_\tau \big(e^{\mu(\tau)} (Y(\tau) -c) \big)^{-1}  \geq \big[(1-4 \gamma) \mathcal{N}(\tau)  +\mathcal{M}(\tau)\big]e^{-\mu(\tau)} (Y(\tau)-c)^{-2},
\tag{$\ast$}
\end{equation}
where
\begin{align*}
\mathcal{M}(\tau) & := 3 \gamma \mathcal{N}(\tau) - ( Y(\tau) - c) \partial_\tau \mu(\tau) -\frac{K_0 }{4\beta(t-t_s)}- \frac{1}{2 \sqrt{\beta(t-t_s)}}\big[A_0(\tau) +A_2(\tau) \big ] - A_3(\tau).
\end{align*}

Take $\gamma:=\frac{1}{8}$.

\begin{lemma} $\mathcal{M}(\tau) \geq 0$ for all $\tau \in [t_s, (t+t_s)/2]$, for $c$ sufficiently large,  $\nu \leq \frac{\sigma}{8}$.
\end{lemma}

\begin{proof}[Proof of Lemma]
Recall $\varphi=\Gamma(-\log V)$. Then
\[
\nabla \varphi = \bigg( \frac{\nabla \Gamma}{\Gamma} +\frac{\nabla \log V}{\log V} \bigg) \varphi, 
\]
and so
\begin{align*}
\frac{(\nabla \varphi)^2}{\varphi} & = \bigg( \frac{\nabla \Gamma}{\Gamma} +\frac{\nabla (-\log V)}{-\log V} \bigg)^2 \varphi \\
& \leq 2 \bigg( \frac{(\nabla \Gamma)^2}{\Gamma} (-\log V) +\frac{(\nabla \log V)^2}{-\log V} \Gamma \bigg) \\
& \leq 2 \bigg( \frac{|o-\cdot|^2}{(2\beta(t-t_s))^2}\Gamma (-\log V) + \Gamma (\nabla \log V)^2 \bigg) \;\;(\text{since } - \log V >1).
\end{align*}
Using the identity $\frac{|o-\cdot|^2}{4 \beta (t-t_s)} \Gamma = \beta (t-t_s) \Delta \Gamma +\frac{d}{2} \Gamma,$ we obtain
\begin{align*}
\frac{1}{2}\bigg \langle \frac{(\nabla \varphi)^2}{\varphi} \bigg \rangle & \leq \big \langle \Gamma (\nabla \log V )^2 \big \rangle + \frac{1}{\beta (t-t_s)}\big \langle \big(\beta (t-t_s) \Delta \Gamma + \frac{d}{2} \Gamma \big)(-\log V) \big \rangle \\
& \leq \sigma^{-1} \mathcal{N} +\big \langle \nabla \Gamma, \nabla \log V \big \rangle + \frac{d}{2 \beta(t-t_s) } \langle \varphi \rangle \\
& \leq 2 \sigma^{-1} \mathcal{N} +\frac{1}{4} \bigg\langle \frac{(\nabla \Gamma)^2}{\Gamma} \bigg \rangle + \frac{d}{2 \beta (t-t_s)} \langle \varphi \rangle\\
& \leq 2 \sigma^{-1} \mathcal{N} + \frac{d}{8 \beta (t-t_s)}+ \frac{d}{2 \beta (t-t_s)} \langle \varphi \rangle.
\end{align*}

Thus, by \eqref{divb_fbd},
\begin{align*}
A_3(\tau) & \leq \frac{\nu}{4}\left\langle \frac{(\nabla \varphi)^2}{\varphi} \right\rangle  + \lambda \nu \langle \varphi \rangle \\
& \leq  \sigma^{-1} \nu\mathcal{N} + \frac{d\nu}{16 \beta (t-t_s)}+ \frac{d\nu + 2 \beta \lambda \nu}{4 \beta (t-t_s)} \langle \varphi \rangle.
\end{align*}

We estimate $A_0(\tau)$ and $A_2(\tau)$ as in \cite{S2}. For the sake of completeness, we provide the details.
Using the inequalities $(B+C+D)^{1/2}\leq (B+D)^{1/2}+C^{1/2}$ and $E^{1/2}(B+D)^{1/2}M^{1/2} \leq (B+D)\varepsilon + (4 \varepsilon)^{-1} E M$ for positive numbers with $\varepsilon = \sigma \gamma/2,$ we obtain
\begin{align*}
\frac{A_0(\tau)}{2 \sqrt{\beta (t-t_s)}} & \leq \frac{\delta}{4 \sqrt{\beta (t-t_s)}} \bigg( 2\sigma^{-1} \mathcal{N}(\tau)+ \frac{d}{8 \beta (t-t_s)}+ \frac{d}{2 \beta (t-t_s)} \langle \varphi \rangle \bigg)^{1/2}\langle \varphi \rangle^{1/2} + \frac{\sqrt{c(\delta)}}{2 \sqrt{\beta (t-t_s)}}\langle \varphi \rangle \\
& \leq \gamma \mathcal{N}(\tau) + \frac{c^*_0}{2 \beta (t-t_s)} \langle \varphi \rangle + \frac{\sigma \nu d}{16 \beta (t-t_s)},
\end{align*}
where $c^*_0 =c^*_0(d,\sigma, \xi, \delta, c(\delta), \gamma) > 0.$

Analogous calculation shows
\[
\frac{A_2(\tau)}{2 \sqrt{\beta (t-t_s)}} \leq \gamma \mathcal{N}(\tau) + \frac{c^*_2}{2 \beta (t-t_s)} \langle \varphi \rangle + \frac{\sigma \gamma d}{16 \beta (t-t_s)},
\]
where $c^*_2 =c^*_2(d,\sigma, \xi, \delta, c(\delta), \gamma) > 0$. Let us only note that in order to estimate $\big \langle (\nabla \psi)^2/ \psi \big \rangle$ in the same way as $\big \langle (\nabla \varphi)^2/ \varphi \big \rangle$ we need the inequality
\[
\bigg \langle \frac{(\nabla \log V)^2}{-\log V} \frac{|o-\cdot|^2}{4 \beta (t-t_s)} \Gamma \bigg \rangle \leq \big \langle \Gamma (\nabla \log V)^2 \big \rangle
\]
which is valid since $-\log V  > \frac{|o-\cdot|^2}{4 \beta (t-t_s)}$ (the inequality $-\log V \geq 1$ would not be enough). The latter is the reason why in the definition of $G_\epsilon (\tau)$ we have ``$\epsilon \Gamma$'' rather than simply ``$\epsilon$'' (as in the proofs of Theorems \ref{second_thm} and 3A).

Thus, we obtain
\begin{align*}
& \frac{1}{2 \sqrt{\beta(t-t_s)}} \big[ A_0(\tau) +A_2(\tau) \big ] + A_3(\tau) \\
& \leq (2 \gamma+\sigma^{-1}\nu) \mathcal{N}(\tau) + \frac{\sigma \gamma d + \frac{d\nu}{2}}{8 \beta (t-t_s)} + \bigg( \frac{c_0^*+c_2^* + \frac{d\nu}{4}+\frac{\beta\lambda\nu}{2}}{\beta (t-t_s)}\bigg) \langle \varphi \rangle.
\end{align*}
By our assumption, $\nu \leq \sigma \gamma$. Thus,
\begin{align*}
\mathcal{M}(\tau) \geq &  - \frac{ 2K_0 + \sigma \gamma d + \frac{d\nu}{2}}{8 \beta (t-t_s)} - ( Y(\tau) - c) \partial_\tau \mu(\tau)\\
& - \bigg( \frac{c_0^*+c_2^*  + \frac{d\nu}{4}+\frac{\beta\lambda\nu}{2}}{\beta (t-t_s)} \bigg) \langle \varphi \rangle.
\end{align*}
Set $c^*=\frac{2K_0 + \sigma \gamma d + \frac{d\nu}{2}}{8}$. Recalling that $ \partial_\tau \mu(\tau) =\frac{K}{\beta (t-t_s)}$ and fixing $K$ by $K=c_0^*+c_2^*  + \frac{d\nu}{4}+\frac{\beta\lambda\nu}{2}$, we conclude that
\begin{align*}
\mathcal{M}(\tau) & \geq \frac{c}{2} \frac{K-\frac{2}{c}c^*}{\beta (t-t_s)} + \bigg(\frac{c}{2}- Y(\tau) - \langle \varphi \rangle \bigg) \frac{K}{\beta (t-t_s)}.
\end{align*}
Now, 
\begin{align*}
\langle \varphi \rangle  = \big \langle \Gamma (- \log V) \big \rangle & = -\big \langle \Gamma\log \big[ \epsilon \Gamma + U \big ] \big \rangle -\big \langle \Gamma \big \rangle \log \bigg[c_0 (t-t_s)^{d/2} \bigg] \\
& = - G_\epsilon (\tau) - \log \bigg[c_0 (t-t_s)^{d/2} \bigg],
\end{align*}
 or $\langle \varphi \rangle = - Y(\tau) + \frac{d}{2} \log \frac{t-\tau}{t-t_s} - \log c_0,$ and so $-Y(\tau)-\langle \varphi \rangle \geq \log c_0\geq \log \frac{(4 \pi c_4 )^{d/2}}{2e\hat{c}_3}-\frac{1}{4c_4}$.
To end the proof, it remains to select $c$ sufficiently large.
\end{proof}

We now return to $(\ast).$ Recall that $\gamma = 1/8.$ Since $\mathcal{N} \geq \sigma \mathcal{N}_1, \mathcal{N}_1 := \big \langle \Gamma |\nabla \log V |^2 \big \rangle,$ Lemma yields
\[
\partial_\tau \big(e^{\mu(\tau)} (Y(\tau) -c) \big)^{-1}  \geq \frac{\sigma}{2} \mathcal{N}_1(\tau) e^{-\mu(\tau)} (Y(\tau)-c)^{-2}.
\tag{$\ast \ast $}
\]
By the Spectral gap inequality,
\begin{align*}
\mathcal{N}_1 & \geq \frac{1}{2 \beta (t-t_s)} \big \langle \Gamma |\log V -\langle \Gamma \log V \rangle |^2 \big \rangle \\
& = \frac{1}{2 \beta (t-t_s)} \big \langle \Gamma |\log \big[\epsilon \Gamma + U \big] -\langle \Gamma \log \big[\epsilon \Gamma + U \big] \rangle |^2 \big \rangle \\
& \equiv \frac{1}{2 \beta (t-t_s)} \big \langle \Gamma |\log \big[\epsilon \Gamma + U \big] - G_\epsilon |^2 \big \rangle.
\end{align*}
Note that $\frac{1}{2}|o-\cdot|^2\leq |z-\cdot|^2 +|o-z|^2$. Clearly, $\frac{1}{t-t_s} \leq \frac{1}{t-\tau} \leq \frac{2}{t-t_s}$ and $|z-o| \leq \sqrt{t-t_s}$ combined imply that $-\frac{|z-\cdot|^2}{4c_4(t-\tau)}\leq -\frac{|o-\cdot|^2}{8c_4(t-t_s)} + \frac{|o-z|^2}{2c_4(t-t_s)},$ and hence
$$
k_{c_4}(t,z;\tau, \cdot) \leq 2^\frac{d}{2} e^\frac{1}{2c_4} k_{2c_4} (t,o;t_s,\cdot).
$$
Therefore, by \eqref{ub_11}, $U \leq \hat{c}_3 k_{c_4}(t,z;\tau,\cdot)$ and $\beta = 2 c_4$, 
$$\Gamma \geq C U, \quad  C^{-1}=\hat{c}_3 2^d e^\frac{1}{\beta},$$
\[
\mathcal{N}_1 \geq \frac{C}{2 \beta (t-t_s)} \big \langle U |\log \big[\epsilon \Gamma + U \big] - G_\epsilon |^2 \big \rangle,
\]
and so, by $\langle U \rangle = 1,$
\[
\mathcal{N}_1 \geq \frac{C}{2 \beta (t-t_s)} \big \langle U |\log \big[\epsilon \Gamma + U \big] - G_\epsilon | \big \rangle^2.
\]
Now,
\begin{align*}
\big \langle U |\log \big[\epsilon \Gamma + U \big] - G_\epsilon | \big \rangle & \geq \big \langle U \log \big[\epsilon \Gamma + U \big] \big \rangle - G_\epsilon \big \langle U \big \rangle \\
& \geq \big \langle U \log U \big \rangle - G_\epsilon \big \langle U \big \rangle \\
& \geq - G_\epsilon(\tau) - \tilde{Q}(t-\tau) - \mathcal{C} \\
& \equiv - Y(\tau) - \mathcal{C}.
\end{align*}
Here we again have used $\langle U \rangle = 1$ and the Nash entropy estimate $-\big \langle U \log U \big \rangle \leq \tilde{Q}(t-\tau) + \mathcal{C}.$
(We note that this simple estimate requires a proof: use $e^\frac{Q}{d}\leq CM$, see Claim \ref{_cl_6} in Section \ref{fbd_2_subsect} and, by \eqref{ub_12}, $M\leq C\sqrt{t-t_s}$.)

\textit{Case} (a): For all $\tau \in \big[t_s, \frac{t+t_s}{2}\big],$
\[
-Y(\tau) -c -2 \mathcal{C} \geq 0.
\]
Here $c$ is from $(\ast \ast)$. Then $- Y(\tau) - \mathcal{C} \geq \frac{1}{2} (-Y(\tau) +c ) > \mathcal{C} > 0$ and hence
\[
\mathcal{N}_1(\tau) \geq  \frac{C}{8 \beta (t-t_s)} \big(-Y(\tau) + c \big)^2.
\]
Thus, by $(\ast \ast),$ $$\big( c - Y(t_s) \big)^{-1} \geq \frac{\sigma C}{16 (t-t_s)} e^{-\mu(t_s)} \int_{t_s}^{(t+t_s)/2} e^{- \mu(\tau)} d \tau \geq \frac{\sigma C e^{3 K/(4\beta)}}{16 (t-s)} \int_{t_s}^{(t+t_s)/2} d \tau,$$ and so $$c - Y(t_s) \leq \frac{32}{\sigma C e^{3 K/(4\beta)}} \leq \frac{2^{d +5} \hat{c}_3}{\sigma},$$ or $G_\epsilon (t_s) \geq -\tilde{Q}(t-t_s) +c - \frac{2^{d +5} \hat{c}_3}{\sigma}.$

\textit{Case} (b): For some $\tau \in \big[t_s, \frac{t+t_s}{2}\big],$
\[
-Y(\tau) -c -2 \mathcal{C} < 0.
\]
By $(\ast \ast),$ $$\big(e^{\mu(\tau)} (Y(\tau) -c) \big)^{-1} \geq \big(e^{\mu(t_s)} (Y(t_s) -c) \big)^{-1},$$ or $$c -Y(t_s) \leq e^{\mu(\tau)-\mu(t_s)} (c - Y(\tau)).$$ Therefore,
$$c-Y(t_s) \leq e^{\mu(\tau)-\mu(t_s)} 2 (c + \mathcal{C}) \leq e^{\frac{K}{4 \beta}} 2 (c + \mathcal{C}),$$ or $G_\epsilon (t_s) \geq -\tilde{Q}(t-t_s) +c -  e^{\frac{K}{4 \beta}} 2 (c + \mathcal{C}).$ 
\end{proof}

\medskip

\subsection{$G$-bound for $-\nabla \cdot a \cdot \nabla + \nabla \cdot b$}

Set $\Lambda_*=A + \nabla \cdot b$, $A=-\nabla \cdot a \cdot \nabla.$
Let $u_*(t,x;s,y)$ denote the heat kernel of $\Lambda_*$.
By Theorem 2B, by duality, $u_*(t,x;s,y)$ satisfies the Gaussian upper bound
\begin{equation}
\label{ub_12}
\tag{${\rm UGB}_{u_*}$}
u_*(t,x;s,y) \leq \hat{c}_3 k_{c_4}(t-s;x-y), \quad 0 < t-s \leq 1,
\end{equation}
for generic constants $\hat{c}_3$, $c_4$. 

The next proposition is valid under weaker assumptions than those in Theorem 3B, that is, it suffices to assume \eqref{ub_12} and 
$$
\|({\rm div\,}b_+)^{\frac{1}{2}}f\|_2^2 \leq \nu \|\nabla f\|_2^2 + \lambda \nu \|f\|_2^2, \quad f \in W^{1,2}.
$$

\begin{proposition}
\label{g2_thm_3}
Let $\beta$ and $\mathbb{C}$ be (generic) constants defined in Proposition \ref{g1_thm_3}. 
Set $o=\frac{x+y}{2}$, $x,y \in \mathbb R^d$, $t_s=\frac{t+s}{2}$.
Then
\[
G(t_s) 
:=\langle k_\beta(t_s-s,o -\cdot)\log u_*(t_s, \cdot; s,z) \geq - \tilde{Q}(t_s-s) - \mathbb{C}, \quad z \in B(o,\sqrt{t_s-s}).
\]
\end{proposition}

\begin{proof}
The proof repeats the proof of Proposition \ref{g1_thm_3}, except that we have to deal with the positive part ${\rm div\,}b_+$ of the divergence of $b$.
\end{proof}

\medskip

Armed with Propositions \ref{g1_thm_3} and \ref{g2_thm_3}, we can repeat the argument in Sections \ref{fbd_5_subsect} and \ref{fbd_8_subsect}, using the assumption ${\rm div\,}b \in \mathbf{K}^d_\nu$. This ends the proof of Theorem 3B.

\bigskip

\appendix

\section{Extrapolation Theorem}

\label{app_B}

\begin{theorem}[{T.\,Coulhon-Y.\,Raynaud}]
\label{thm_cr}
Let $U^{t,s}: L^1 \cap L^\infty \rightarrow L^1 + L^\infty$ be a two-parameter evolution family of operators:
\[U^{t,s} = U^{t,\tau}U^{\tau,s}, \quad 0 \leq s < \tau < t \leq \infty.
\]
Suppose that, for some $1 \leq p < q < r \leq \infty,$ $\nu > 0,$ $M_1$ and $M_2,$ the inequalities
\[
\| U^{t,s} f \|_p \leq M_1 \| f \|_p \quad \text{ and } \quad \| U^{t,s} f \|_r \leq M_2 (t-s)^{-\nu} \|  f \|_q
\]
are valid for all $(t,s)$ and $f \in L^1 \cap L^\infty.$ Then
\[
\| U^{t,s} f \|_r \leq M (t-s)^{-\nu/(1-\beta)} \| f \|_p ,
\]
where $\beta = \frac{r}{q}\frac{q-p}{r-p}$ and $M = 2^{\nu/(1-\beta)^2} M_1 M_2^{1/(1-\beta)}.$
\end{theorem}

For the proof see e.g.\,\cite[Appendix F]{KiS2}.

%
%

\end{document}